\documentclass[12pt]{amsart}
\usepackage{amssymb, amsmath, mathtools}

\usepackage[dvipsnames]{xcolor}
\usepackage{color}
\usepackage[utf8]{inputenc}
\usepackage[T1]{fontenc}
\usepackage{fullpage}
\usepackage[normalem]{ulem}
\usepackage{amsmath,amscd,mathrsfs}
\usepackage{tikz}
\usepackage[colorlinks=true, citecolor=blue, linkcolor=red,pagebackref, hyperindex]{hyperref}
\usepackage{mathtools}

\newtheorem{Theorem}{Theorem}[section]
\newtheorem{Theoremx}{Theorem}
 
\newtheorem{Potential Theorem}[Theorem]{Potential Theorem}
\newtheorem{Lemma}[Theorem]{Lemma}
\newtheorem{Corollary}[Theorem]{Corollary}
\newtheorem{Proposition}[Theorem]{Proposition}
\theoremstyle{definition}
\newtheorem{Example}[Theorem]{Example}
\newtheorem{Condition}[Theorem]{Condition}

\newtheorem{Definition}[Theorem]{Definition}
\newtheorem{Question}[Theorem]{Question}

\theoremstyle{remark} 
\newtheorem{Remark}[Theorem]{Remark}

\DeclareMathOperator{\height}{ht}
\DeclareMathOperator{\heigh}{ht}
\DeclareMathOperator{\Max}{Max}

\DeclareMathOperator{\Supp}{Supp}
\DeclareMathOperator{\Ann}{Ann}

\DeclareMathOperator{\Hom}{Hom}
\DeclareMathOperator{\Spec}{Spec}
\DeclareMathOperator{\Min}{Min}
\DeclareMathOperator{\Ass}{Ass}

\DeclareMathOperator{\im}{im}

\DeclareMathOperator{\rank}{rank}
\DeclareMathOperator{\minrank}{min-rank}
\DeclareMathOperator{\frk}{frk}

\DeclareMathOperator{\e}{e_{HK}}
\DeclareMathOperator{\te}{e^{loc}_{HK}}
\DeclareMathOperator{\s}{s}
\DeclareMathOperator{\ts}{\s_{loc}}

\DeclareMathOperator{\Assh}{Assh}

\def\char{\mbox{char}\,}

\def\p{\mathfrak{p}}

\def\m{\mathfrak{m}}
\def\n{\mathfrak{n}}
\def\a{\mathfrak{a}}

\def\R{\mathbb{R}}
\def\Q{\mathbb{Q}}
\def\C{\mathscr{C}}
\def\F{\mathbb{F}}

\def\D{\mathscr{D}}
\def\E{\mathscr{E}}
\def\N{\mathbb{N}}
\def\O{\mathcal{O}}

\def\ds{\displaystyle}

\renewcommand{\ge}{\geqslant}\renewcommand{\geq}{\geqslant} 
\renewcommand{\leq}{\leqslant}

\newcommand{\blank}{\underline{\hskip 10pt}}

\begin{document}

\title{Globalizing F-invariants}
\author{Alessandro De Stefani}
\address{Department of Mathematics, Royal Institute of Technology (KTH), Stockholm, 100 44, Sweden}
\email{ads@kth.se}
\author{Thomas Polstra}
\address{Department of Mathematics, University of Missouri-Columbia, Columbia, MO 65211}
\email{tmpxv3@mail.missouri.edu}
\author{Yongwei Yao}
\address{Department of Mathematics and Statistics, Georgia State University, Atlanta, Georgia 30303, USA}
\email{yyao@gsu.edu}

\begin{abstract} In this paper we define and study the global Hilbert-Kunz multiplicity and the global F-signature of prime characteristic rings which are not necessarily local.  Our techniques are made meaningful by extending many known theorems about Hilbert-Kunz multiplicity and F-signature to the non-local case.
\end{abstract}

\maketitle

\section{Introduction}

Throughout, $R$ will be a commutative Noetherian ring with identity. Unless otherwise stated, $R$ is of prime characteristic $p$. Let $F^e: R\rightarrow R$ be the $e$th iterate of the Frobenius endomorphism, that is $F^e(r)=r^{p^e}$.  Kunz's work in \cite{Kunz1969} equates flatness of $F^e$ with the property that $R$ is regular, a foundational result indicating asymptotic measurements of the Frobenius endomorphism can be used to measure the severity of the singularities of $R$.  We will focus on the numerical invariants Hilbert-Kunz multiplicity and F-signature. 

For the sake of simplicity in introducing Hilbert-Kunz multiplicity and F-signature, assume that $(R,\m,k)$ is a complete local domain, with unique maximal ideal $\m$, dimension $d$, residue field $k$, and $k^{1/p}$ is finite as a $k$-vector space. If $I\subseteq R$ is an ideal, $I^{[p^e]}=(i^{p^e}\mid i\in I)$ is the expansion of $I$ along $F^e$. If $M$ is a finite length $R$-module let $\lambda(M)$ denote the length of $M$. If $I$ is an $\m$-primary ideal, so is $I^{[p^e]}$ for each $e\in\N$. 

\begin{Definition}\label{Definition of HK multiplicity} Let $(R,\m,k)$ be a local ring of prime characteristic $p$ and $I$ an $\m$-primary ideal. The \emph{Hilbert-Kunz multiplicity of $I$} is $$\e(I)=\lim_{e\rightarrow \infty}\lambda(R/I^{[p^e]})/p^{e\dim(R)}.$$
\end{Definition}

Monsky proved the existence of the limit $\e(I)$ in \cite{Monsky1983}. The Hilbert-Kunz multiplicity of $R$ is defined to be $\e(R)=\e(\m)$. It is well known that $\e(R)\geq 1$ with equality if and only if $R$ is regular, \cite{WatanabeYoshida2000}. More generally, it is known that sufficiently small values of Hilbert-Kunz multiplicity imply the properties of being Gorenstein and strongly F-regular, \cite{BlickleEnescu, AberbachEnescu2008}.

Denote by $F^e_*R$ the $R$-module obtained via restriction of scalars via $F^e$. Our hypotheses imply that $R$ is F-finite, that is, $F^e_*R$ is a finitely generated $R$-module for each $e\in\N$. Moreover, we have that $\lambda(R/\m^{[p^e]}R)/p^{ed}=\mu(F^e_*R)/\rank(F^e_*R)$, where $\mu(\blank)$ denotes the minimal number of generators of a finitely generated $R$-module. In particular, $$\e(R)=\lim_{e\rightarrow \infty}\mu(F^e_*R)/\rank(F^e_*R).$$ Thus, the Hilbert-Kunz multiplicity of $R$ is the asymptotic growth rate of the minimal number of generators of $F^e_*R$ compared to its rank, a measurement that can also be discussed for rings that are not necessarily local.

Now let $R$ be an F-finite domain, not necessarily local. With the above observation, we define the \emph{global Hilbert-Kunz multiplicity} of $R$, still denoted $\e(R)$, as 
\[
\e(R)=\lim_{e\rightarrow\infty}\mu(F^e_*R)/\rank(F^e_*R),
\]
provided the limit exists. Our first main result is the existence of the corresponding limit for any F-finite ring.  In addition, we relate $\e(R)$ with the Hilbert-Kunz multiplicities $\e(R_P)$ of the localizations at primes $P \in \Spec(R)$, showing that such an invariant, even though it is defined globally, captures the local properties of the ring. Finally, as for the Hilbert-Kunz multiplicity of a local ring, we show that small values of $\e(R)$ imply that $R$ has mild singularities. We summarize all these results in the following theorem. We point out that our results hold in a more general setup than the one in which we state them here, as we will show in Section~\ref{Global HK}.

\begin{Theoremx} Let $R$ and $T$ be F-finite domains, not necessarily local.
\begin{enumerate}
\item(Theorem~\ref{HK exists}) The limit $\e(R)=\lim_{e\rightarrow \infty }\mu(F^e_*R)/\rank(F^e_*R)$ exists.
\item (Theorem~\ref{What is HK}) We have  $\e(R)=\max\{\e(R_P)\mid P\in\Spec(R)\}$.
\item (Theorem~\ref{Watanabe-Yoshida Hilbert-Kunz}) The ring $R$ is regular if and only if $\e(R)=1$.
\item (Theorem~\ref{HK small means good}) There exists a positive real number $\delta$ such that, if $\e(R) \leq 1+ \delta$, then $R$ is strongly F-regular and Gorenstein.
\item (Theorem~\ref{HK Flat Extension Global}) If $R\rightarrow T$ is faithfully flat, then $\e(R)\leq \e(T)$.
\end{enumerate}
\end{Theoremx}

We now turn our attention to the F-signature. To introduce it, we return to the assumptions that $(R,\m,k)$ is a complete local domain of dimension $d$, and $k^{1/p}$ is a finite $k$-vector space. As noted before, these assumptions guarantee that $R$ is F-finite. We denote by $\frk(F^e_*R)$ the maximal number of free $R$-summands appearing in all direct sum decompositions, equivalently the maximal number of free $R$-summands appearing in a single direct sum decomposition, of $F^e_*R$ into indecomposable modules. 

\begin{Definition}\label{Definition of F-signature} Let $(R,\m,k)$ be local domain of prime characteristic $p$ and assume that $R$ is F-finite. The \emph{F-signature of $R$} is $$\s(R)=\lim_{e\rightarrow\infty}\frk(F^e_*R)/\rank(F^e_*R).$$
\end{Definition}

Tucker proved the existence of $\s(R)$ in \cite{Tucker2012}. Before Tucker's proof of the existence of $\s(R)$, the study of the asymptotic growth rate of the number of free summands of $F^e_*R$ originated in \cite{SmithVandenBergh}. Huneke and Leuschke coined the term F-signature in \cite{HunekeLeuschke} and were able to show it exists under the additional assumption that $R$ is Gorenstein.  There were a number of papers written which established the existence of the F-signature for certain classes of rings, see \cite{HunekeLeuschke}, \cite{Singh2005}, \cite{AberbachEnescu2006}, \cite{Yao2006}, and \cite{Aberbach2008}. As remarked by third author in \cite{Yao2006}, the study of F-signature is closely related to relative Hilbert-Kunz multiplicity studied by Watanabe and Yoshida in \cite{WatanabeYoshida2004}. Similar to Hilbert-Kunz multiplicity, particular values of $\s(R)$ determine the severity of the singularity of $R$. Most notably, $\s(R)=1$ if and only if $R$ is regular, as shown by Huneke and Leuschke in \cite{HunekeLeuschke}, and $\s(R)>0$ if and only if $R$ is strongly F-regular by work of Aberbach and Leuschke in \cite{AberbachLeuschke}. 

In order to globalize F-signature, note that the numbers $\frk(F^e_*R)$ make sense also for F-finite rings which are not necessarily local. Unlike the local case, one has to consider all direct sum decompositions of $F^e_*R $ to determine $\frk(F^e_*R)$ and not just a single direct sum decomposition. Nevertheless, it is possible to study the sequence of measurements $\frk(F^e_*R)/\rank(F^e_*R)$ for any F-finite ring. We prove that the limit $\s(R)$ of such a sequence exists, and we call it the global F-signature of $R$. As with global Hilbert-Kunz multiplicity, we relate $\s(R)$ with the local F-signatures $\s(R_P)$, for $P \in \Spec(R)$. In addition, we show that special values of $\s(R)$ detect the singularities of the ring $R$, as in the case of local rings. Our main results about F-signature, here stated for simplicity in a more restrictive setup than the one in which they actually hold, are summarized in the following theorem.

\begin{Theoremx} Let $R$ and $T$ be F-finite domains, not necessarily local.
\begin{enumerate}
\item (Theorem~\ref{Global F-signature exists}) The limit $\s(R) = \lim_{e\rightarrow \infty} \frk(F^e_*R)/\rank(F^e_*R)$ exists. 
\item (Theorem~\ref{What is Global F-signature}) We have $\s(R)=\min\{\s(R_P)\mid P\in\Spec(R)\}.$
\item (Theorem~\ref{s=1})  The ring $R$ is regular if and only if $\s(R)=1$.
\item (Theorem~\ref{s>0}) The ring $R$ is strongly F-regular if and only if $\s(R)>0$.
\item (Theorem~\ref{Global F-signature flat}) If $R\rightarrow T$ is faithfully flat, then $\s(R)\geq \s(T)$.
\end{enumerate}
\end{Theoremx}

This article is organized as follows. Section~\ref{Background} is used to set up notation and recall previously known results. Section~\ref{Global HK} develops the theory of global Hilbert-Kunz multiplicity of finitely generated modules over non-local F-finite rings. We then introduce the theory of global F-signature of finitely generated $R$-modules over non-local F-finite rings in Section~\ref{Global F-signature section}. The global F-signature of a pair $(R,\D)$ where $R$ is an F-finite ring and $\D$ is a Cartier subalgebra is also introduced.  In section~\ref{Global non-F-finite} we study global Hilbert-Kunz multiplicity and global F-signature of faithfully flat extensions. Besides showing similarities between the local and global theory of Hilbert-Kunz multiplicity and F-signature, we also provide examples in Sections~\ref{Global HK}, \ref{Global F-signature section}, and~\ref{Global non-F-finite} which illustrate their differences.

This paper provides a framework to globalize numerical invariants of positive characteristic rings. In \cite{DSPY2}, the authors of this paper establish results similar to those of this paper for other numerical invariants, including Frobenius Betti numbers. In doing so, we globalize more F-invariants of interest.

\section{Background}\label{Background}

 If $R$ is a domain and $M$ a finitely generated $R$-module, the rank of $M$ is defined as $\rank_R(M)=\dim_K(M\otimes_R K)$, where $K$ is the fraction field of $R$. When $R$ is not a domain, the notion of rank is not necessarily uniquely defined. In particular, in this article we will need to use two different definitions. Given a finitely generated $R$-module $M$, we define the rank of $M$ as $\rank_R(M)=\max\{\rank_{R/Q}(M/QM)\mid Q\in\Min(R)\}$, and we define the min-rank of $M$ as $\minrank_R(M) = \min\{\rank_{R/Q}(M/QM)\mid Q\in\Min(R)\}$. The reason for giving the name of rank to the maximum of the ranks modulo minimal primes is that this is the definition that we will mostly use in this article. Clearly, the two notions agree when $R$ is a domain. As discussed in the introduction, we use $\lambda_R(\blank)$ and $\mu_R(\blank)$ to denote the length of a finite length $R$-module and the minimal number of generators of a finitely generated $R$-module respectively. If confusion is not likely to arise, we commonly omit subscripts from these notations.

\subsection{F-finite rings}  As discussed in the introduction, $R$ is F-finite if for some, equivalently for all positive integers $e\in\N$,  $F^e_*R$ is a finitely generated $R$-module. Every F-finite ring is excellent, \cite[Theorem 2.5]{Kunz1976}. If $R$ is F-finite and $M$ a finitely generated $R$-module, then $F^e_*M$ is a finitely generated $R$-module for each $e\in\N$. Once again, $F^e_*M$ is the $R$-module $M$ obtained via restriction of scalars by $F^e$. If $R$ is a domain, then $F^e_*R$ is naturally isomorphic with $R^{1/p^e}$, the ring of $p^e$th roots of $R$, as $R$ lies in an algebraic closure of its fraction field. However, we will refrain from using this notation and henceforth use $F^e_*\blank$.

Let $R$ be an F-finite ring. Given $P\in\Spec(R)$ let $\kappa(P)=R_P/PR_P$ be the residue field of $R_P$ and let $\alpha(P)=\log_{p^e}[F^e_*\kappa(P):\kappa(P)]$, which is independent of the choice of $e>0$. Let $\gamma(R)=\max\{\alpha(Q)\mid Q\in \min(R)\}$. It is easily verified that, if $R$ is a domain, then $\rank(F^e_*R)=p^{e\gamma(R)}$ for each $e\in\N$.  If $M$ is a finitely generated $R$-module, $I=\Ann_R(M)$, we define $\gamma(M)$ as $\gamma(R/I)$.

 Kunz showed that if $R$ is locally equidimensional F-finite ring, then $\height(P)+\alpha(P)$ is constant on connected components of $\Spec(R)$. We record this result for future reference.

\begin{Lemma}[{\cite[Proposition 2.3]{Kunz1976}\label{Kunz's Lemma}}] Suppose that $R$ is an F-finite ring, then for any two prime ideals $P\subseteq Q$, we have $\alpha(P)=\alpha(Q)+\dim(R_Q/PR_Q)$. In particular, if $R$ is locally equidimensional, then for any two prime ideals $P$ and $Q$ which lie in a common connected component of $\Spec(R)$, we have $\alpha(P)+\height(P)=\alpha(Q)+\height(Q)$.
\end{Lemma}

\subsection{Hilbert-Kunz multiplicity}\label{HK multiplicity} Suppose $R=(R,\m,k)$ is a local ring of prime characteristic $p$, of dimension $d$, and $M$ a finitely generated $R$-module. Let $I\subseteq R$ be an ideal. If $I=(i_1,\dots,i_s)$, then one easily checks that $I^{[p^e]}=(i_1^{p^e},\dots,i_s^{p^e})$. So for each $e\in\N$ there are inclusions of ideals $I^{sp^e}\subseteq I^{[p^e]}\subseteq I^{p^e}$. Therefore if $I$ is $\m$-primary, so is each $I^{[p^e]}$, and we have the set of inequalities $$\lambda(M/I^{p^e}M)\leq \lambda(M/I^{[p^e]}M)\leq \lambda(M/I^{sp^e}M).$$ So as a function, $\lambda(M/I^{[p^e]}M)=O(p^{e\dim(M)})$. Monsky proved the following. 

\begin{Theorem}[{\cite[Theorem 1.8]{Monsky1983}}]
\label{HK exists local} Let $(R,\m,k)$ be a local ring of prime characteristic $p$, dimension $d$, and $M$ a finitely generated $R$-module. Then the following limit exists, $$\lim_{e\rightarrow \infty}\frac{1}{p^{ed}}\lambda(M/I^{[p^e]}M).$$ Its limit is denoted $\e(I,M)$, and is called the Hilbert-Kunz multiplicity of $I$ with respect to $M$. Moreover, $\lambda(M/I^{[p^e]}M)=\e(I,M)p^{ed}+O(p^{e(d-1)}).$
\end{Theorem}

We let $\e(M)=\e(\m, M)$ and call this number the Hilbert-Kunz multiplicity of $M$. Hilbert-Kunz multiplicity is additive on short exact sequences. So if $I$ is an $\m$-primary ideal and $0\rightarrow M'\rightarrow M\rightarrow M''\rightarrow 0$ a short exact sequence of finitely generated $R$-modules, then $\e(I,M)=\e(I,M')+\e(I,M'')$, see \cite[Theorem 1.8]{Monsky1983}. Because of this, study of the  Hilbert-Kunz multiplicity of a finitely generated $R$-module can typically be reduced to the scenario that $R$ is a domain and $M=R$. An application of the additivity of Hilbert-Kunz multiplicity is the associativity formula. Let $\Assh(R)=\{P\in \Spec(R)\mid \dim(R/P)=\dim(R)\}$ for a local ring $(R,\m)$.

\begin{Theorem}[Associativity Formula]\label{Associativity Formula Local} Let $(R,\m,k)$ be a local ring of prime characteristic $p$ and dimension $d$. Let $I$ be an $\m$-primary ideal and $M$ a finitely generated $R$-module. Then $$\e(I,M)=\sum_{P\in\Assh(R)}\lambda_{R_P}(M_P)\e(I, R/P).$$
\end{Theorem}

There are theorems which relate values of $\lambda(R/\m^{[p^e]}R)$ and $\e(R)$ with the severity of the singularity of $(R,\m,k)$, the first of which is Kunz's Theorem.

\begin{Theorem}[{\cite[Theorem 3.3]{Kunz1969}}]\label{Kunz's Theorem} Let $(R,\m,k)$ be a local ring of prime characteristic $p$ and of dimension $d$. Then for each $e\in\N$, $\lambda(R/\m^{[p^e]}R)\geq p^{ed}$ with equality for some, equivalently for all, positive integers $e\in\N$ if and only if $F^e_*R$ is a flat $R$-module for some, equivalently for all, positive integers $e\in\N$ if and only if $R$ is regular.
\end{Theorem}

In particular, if $R$ is regular then $\e(R)=1$. Watanabe and Yoshida are able to prove under mild hypotheses that the reverse implication of this result is true.

\begin{Theorem}[{\cite[Theorem 1.5]{WatanabeYoshida2000}}\label{Watanabe Yoshida Hilbert-Kunz local}\footnote{See \cite{HunekeYao} for a simpler proof.}] Let $(R,\m,k)$ be a formally unmixed local ring of characteristic $p$. Then $R$ is regular if and only if $\e(R)=1$.
\end{Theorem} 

Recall that $(R,\m,k)$ is formally unmixed if $\dim(\hat{R}/Q\hat{R})=\dim(R)$ for each $Q\in \Ass(\hat{R})$. Theorems of Blickle and Enescu in \cite{BlickleEnescu}, and improvements of their Theorems later given by Aberbach and Enescu in \cite{AberbachEnescu2008}, show that rings with small Hilbert-Kunz multiplicity have decent singularities.

\begin{Theorem}[Blickle-Enescu, Aberbach-Enescu]\label{HK small means good local} Suppose that $(R,\m,k)$ is a formally unmixed local ring of characteristic $p$.  Let $e$ be the Hilbert-Samuel multiplicity of $R$. If $\e(R)\leq 1+\max\left\{1/\dim(R)!, 1/e\right\}$, then $R$ is strongly F-regular and Gorenstein.
\end{Theorem}

Simple computations show that the Hilbert-Kunz multiplicity of a local ring need not be an integer. In fact, Brenner has shown that the Hilbert-Kunz multiplicity of a local ring can even be irrational, see \cite{Brenner2013}. As the Hilbert-Kunz multiplicity of a local ring $(R,\m,k)$ is at least $1$, it is natural to ask ``how close'' can $\e(R)$ can be to $1$. Initial investigations of such a question were originated by Blickle and Enescu in \cite{BlickleEnescu} and were significantly improved by Aberbach and Enescu in \cite{AberbachEnescu2008} and Celikbas, Dao, Huneke, and Zhang in \cite{CDHZ}.

\begin{Theorem}[Blickle-Enescu, Aberbach-Enescu, Celikbas-Dao-Huneke-Zhang]\label{HK really small means regular local} Fix $d\in\N$. There is a number $\delta>0$ such that if $(R,\m,k)$ is formally unmixed of dimension $d$, of any prime characteristic, and such that $\e(R)\leq 1+\delta$, then $R$ is regular.
\end{Theorem}

Suppose that $(R,\m)\rightarrow (T,\n)$ is flat local homomorphism of local characteristic $p$ rings. Kunz proved in \cite{Kunz1976} for each $e\in\N$, $\lambda_R(R/\m^{[p^e]})/p^{e\dim(R)}\leq \lambda_T(T/\n^{[p^e]})/p^{e\dim(T)}$. Moreover, Kunz shows that equality occurs if $T/\m T$ is regular. In \cite{Kunz1976}, Kunz was unaware that the limit existed as $e\rightarrow \infty$. In fact, Kunz provides an incorrect counterexample to the existence of Hilbert-Kunz multiplicity, \cite[Example 4.3]{Kunz1976}. Nevertheless, the limit exists, hence the following theorem.

\begin{Theorem}[{\cite[Theorem 3.6, Proposition 3.9]{Kunz1976}}]\label{HK Flat Extension} Let $(R,\m,k)\rightarrow (T,\n,l)$ be a flat local ring homomorphism of local rings of prime characteristic $p$. Then for each $e\in\N$,  $\lambda_R(R/\m^{[p^e]})/p^{e\dim(R)}\leq \lambda_T(T/\n^{[p^e]})/p^{e\dim(T)}$, hence $\e(R)\leq \e(T)$. Moreover, if $T/\m T$ is regular, then for each $e\in \N$, $\lambda_R(R/\m^{[p^e]})/p^{e\dim(R)}= \lambda_T(T/\n^{[p^e]})/p^{e\dim(T)}$, hence $\e(R)=\e(T)$.
\end{Theorem}

More generally, suppose that $(R,\m) \to (T,\n)$ is a flat local homomorphism of characteristic $p$ local rings, $M$ a finitely generated $R$-module, and $M_T=M\otimes_R T$. Then Kunz's methods can be extended to show $\lambda_R(M/\m^{[p^e]}M)/p^{e\dim(R)}\leq \lambda_T(M_T/\n^{[p^e]} M_T)/p^{e\dim(T)}$ holds for all $e \in \N$. Furthermore, equality holds if the closed fiber of $R \to T$ is regular. To prove this, we begin with two lemmas.

\begin{Lemma} \label{HK Inequality Same Dim} Let $(R,\m,k)\rightarrow (T,\n,l)$ be a flat local ring homomorphism of local rings of prime characteristic $p$, and $M$ a finitely generated $R$-module. Assume that $\dim(R) = \dim(T)$, and denote $M_T  = M \otimes_R T$. For each $e\in\N$, we have $\lambda_R(M/\m^{[p^e]}M)/p^{e\dim(R)}\leq \lambda_T(M_T/\n^{[p^e]} M_T)/p^{e\dim(T)}$. Moreover, if $T/\m T$ is regular, then $\lambda_R(M/\m^{[p^e]}M)/p^{e\dim(R)}= \lambda_T(M_T/\n^{[p^e]}M_T)/p^{e\dim(T)}$ for all $e \in \N$.
\end{Lemma}
\begin{proof}
By flatness we have $\lambda_T(M_T /\m^{[p^e]} M_T)/p^{e\dim(T)} = \lambda_R(M/\m^{[p^e]} M)/p^{e \dim(R)} \cdot \lambda_T(T/\m T)$. Furthermore, starting from a filtration $\m T = I_0 \subseteq I_i \subseteq \ldots \subseteq I_t = T$ with $I_j/I_{j-1} \cong T/\n$, we get a filtration $\m^{[p^e]} M_T = I_0^{[p^e]} M_T \subseteq I_1^{[p^e]} M_T \subseteq \ldots \subseteq M_T$. Each successive quotient $I_j^{[p^e]}M_T/ I_{j-1}^{[p^e]}M_T$ is a homomorphic image of $M_T/\n^{[p^e]} M_T$. Thus, we obtain
\[
\ds \lambda_T(M_T/\m^{[p^e]} M_T)/p^{e \dim(T)} \leq \lambda_T(M_T/\n^{[p^e]}M_T)/p^{e \dim(T)} \cdot \lambda_T(T/\m T).
\]
Combining these facts gives the desired inequality:
\[
\ds \lambda_R(M/\m^{[p^e]}M)/p^{e\dim(R)} = \frac{\lambda_T(M_T/\m^{[p^e]}M_T)/p^{e\dim(T)}}{\lambda_T(T/\m T)} \leq \lambda_T(M_T/\n^{[p^e]}M_T)/p^{e\dim(T)}.
\]
If $T/\m T$ is regular, then it must be a field, so that $\lambda_T(T/\m T) = 1$, $\m T=\n$ and hence $\m^{[p^e]}T=\n^{[p^e]}$. Hence $$\lambda_T(M_T/\n^{[p^e]}M_T)/p^{e\dim(T)}= \lambda_R(M/\m^{[p^e]} M)/p^{e \dim(R)} \cdot \lambda_T(T/\m T)=\lambda_R(M/\m^{[p^e]} M)/p^{e \dim(R)}.$$
\end{proof}

The following lemma is the extension of \cite[Corollary 3.8]{Kunz1976} to the module case. We also refer the reader to \cite[Theorem 3.3]{HunekeYao}.

\begin{Lemma} \label{HK Inequality Localization} Let $(R,\m,k)$ be a local ring of prime characteristic $p$, and $P \in \Spec(R)$ such that $\height(P) + \dim(R/P) = \dim(R)$. Then for every finitely generated $R$-module $M$ we have $\lambda_R(M/\m^{[p^e]}M)/p^{e\dim(R)} \geq \lambda_{R_P}(M_P/P^{[p^e]}M_P)/p^{e\dim(R_P)}$.
\end{Lemma}
\begin{proof}
There exists a local flat extension $(T,\n,l)$ of $(R,\m,k)$ such that $T$ is F-finite, and $\m T= \n$. Furthermore, we can choose $Q \in \Spec(T)$ such that $Q \cap R = P$, $\height(Q) = \height(P)$, and $\dim(T/Q) = \dim(R/P)$. In fact, let $Q \in \Min(PT)$ be such that $\dim(T/Q) = \dim(T/PT)$. Also, since $Q \cap R = P$ and $R \to T$ is flat, we have $\height(Q) \geq \height(P)$. Therefore we get
\[
\ds \dim(T) \geq \dim(T/Q) + \height(Q) \geq \dim(T/PT) + \height(P) = \dim(R/P) +  \height(P) = \dim(R),
\]
where we used the fact that $R/P \to T/PT$ is a flat map of local rings of the same dimension. Since $\dim(R) = \dim(T)$, this shows that $\height(P) = \height(Q)$ and $\dim(R/P) = \dim(T/Q)$. Let $M' = M \otimes_R T$, and observe that $\lambda_R(M/\m^{[p^e]}M)/p^{e\dim(R)} = \lambda_T(M'/\n^{[p^e]}M')/p^{e\dim(T)}$ holds for all $e \in \N$ by Lemma~\ref{HK Inequality Same Dim}, since the closed fiber is regular. In addition, since $R_P \to T_Q$ is a flat local homomorphisms of local rings of the same dimension, we have $\lambda_{R_P}(M_P/P^{[p^e]}M_P)/p^{e\dim(R_P)} \leq \lambda_{T_Q}(M'_Q/Q^{[p^e]}M'_Q)/p^{e\dim(T_Q)}$, again by Lemma~\ref{HK Inequality Same Dim}. Thus, it suffices to show that $\lambda_T(M'/\n^{[p^e]}M')/p^{e\dim(T)} \geq \lambda_{T_Q}(M'_Q/Q^{[p^e]}M'_Q)/p^{e\dim(T_Q)}$. Since $T$ is F-finite and $\height(Q) + \dim(T/Q) = \dim(T)$, we have that $\gamma(T) = \gamma(T_Q)$. Therefore 
\[
\ds \frac{\lambda_T(M'/\n^{[p^e]}M')}{p^{e\dim(T)}} = \frac{\lambda_T(F^e_*M'/\n F^e_*M')}{p^{e\gamma(T)}} \geq \frac{\lambda_{T_Q}(F^e_*M'_Q/Q F^e_*M'_Q)}{p^{e\gamma(T_Q)}} = \frac{\lambda_{T_Q}(M'_Q/Q^{[p^e]}M'_Q)}{p^{e\dim(T_Q)}}.
\]
The inequality in the middle follows from the fact that the minimal number of generators of the $T$-module $F^e_*M'$ can only decrease after localization at the prime $Q$.
\end{proof}

\begin{Theorem} \label{HK Inequality} Let $(R,\m,k)\rightarrow (T,\n,l)$ be a flat local ring homomorphism of local rings of prime characteristic $p$, and $M$ a finitely generated $R$-module. Denote $M_T  = M \otimes_R T$. Then for each $e\in\N$ we have  $\lambda_R(M/\m^{[p^e]}M)/p^{e\dim(R)}\leq \lambda_T(M_T/\n^{[p^e]} M_T)/p^{e\dim(T)}$, hence $\e(M) \leq \e(M_T)$. Moreover, if $T/\m T$ is regular, then for each $e\in \N$ we have that $\lambda_R(M/\m^{[p^e]}M)/p^{e\dim(R)}= \lambda_T(M_T/\n^{[p^e]}M_T)/p^{e\dim(T)}$, hence $\e(M) = \e(M_T)$.
\end{Theorem}

\begin{proof}
It is clearly sufficient to prove the statements about lengths, as those regarding Hilbert-Kunz multiplicities follow by taking limits as $e \to \infty$. Because of flatness, there exists a prime $P \in \Spec(T)$ such that $P \cap R = \m$. Furthermore, we can choose $P$ such that $\height(P) = \dim(R)$ and $\height(P) + \dim(T/P) =\dim(T)$, as in the proof of Lemma~\ref{HK Inequality Localization}. Lemma~\ref{HK Inequality Same Dim} applied to the faithfully flat map $R \to T_P$ yields
\[
\ds \lambda_R(M/\m^{[p^e]}M)/p^{e\dim(R)} \leq \lambda_{T_P}(M_{T_P}/P^{[p^e]} M_{T_P})/p^{e \dim(T_P)},
\]
with equality when $T/\m T$ is regular. Lemma~\ref{HK Inequality Localization} applied to the ring $(T,\n,l)$ and the prime $P \in \Spec(T)$ gives 
\[
\ds \lambda_{T_P}(M_{T_P}/P^{[p^e]} M_{T_P})/p^{e \dim(T_P)} \leq \lambda_T(M_T/\n^{[p^e]}M_T)/p^{e\dim(T)}.
\]
Combining the two inequalities, we obtain the first part of the theorem. Now assume that $T/\m T$ is regular. We want to show that the reverse inequality holds. Consider a filtration $0 = J_0 \subseteq J_1\subseteq \cdots \subseteq J_s = M/\m^{[p^e]}M$, with $J_i/J_{i-1} \cong R/\m$ for all $i$. Applying the functor \ $\blank \otimes_R T/\n^{[p^e]}$ we see that $\lambda_T(M_T/\n^{[p^e]}M_T)/p^{e\dim(T)} \leq \lambda_R(M/\m^{[p^e]}M)/p^{e\dim(R)} \cdot \lambda_T(T/\m T+\n^{[p^e]})/p^{e\dim(T/\m T)}$ for all $e \in \N$. Since $T/\m T$ is regular, Theorem~\ref{Kunz's Theorem} guarantees that $\lambda_T(T/\m T+\n^{[p^e]}) = p^{e \dim(T/\m T)}$. Therefore we have $\lambda_T(M_T/\n^{[p^e]}M_T)/p^{e\dim(T)}  \leq \lambda_R(M/\m^{[p^e]}M)/p^{e\dim(R)}$ for all $e \in \N$, as desired. 
\end{proof}

Kunz asked if $R$ is an excellent equidimensional ring of prime characteristic $p$, for each $e\in\N$ is the function $\lambda_e:\Spec(R)\rightarrow \R$ sending $P\mapsto \lambda(R_P/P^{[p^e]}R_P)/p^{e\height(P)}$ upper semi-continuous, see \cite[Problem page 1006]{Kunz1976}. Shepherd-Barron provides a counter-example to Kunz's problem and showed the answer to the question is yes under the stronger assumption $R$ is locally equidimensional, see \cite{SB}. The following theorem is the extension of Shepherd-Barron's result to the module case.

\begin{Theorem}\label{Shepherd-Barron's result for modules} Let $R$ be a locally equidimensional excellent ring of prime characteristic $p$ and let $M$ be a finitely generated $R$-module. Then for each $e\in\N$ the function $\lambda_e:\Spec(R)\rightarrow \R$ sending $P\mapsto \lambda(M_P/P^{[p^e]}M_P)/p^{e\height(P)}$ is upper semi-continuous.
\end{Theorem}

\begin{proof} For each $e\in\N$ and $P\in\Spec(R)$ let $\lambda_e(P)=\lambda(M_P/P^{[p^e]}M_P)/p^{e\height(P)}$. Let $r\in \R$ and $U_r=\{P\in\Spec(R)\mid \lambda_e(P)<r\}$. To show $\lambda_e$ is upper semi-continuous we need to show $U_r$ is open. Let $P\in U_r$. By Lemma~\ref{HK Inequality Localization} and Nagata's criterion for openness, see \cite[Theorem 24.2]{Matsumura1980}, it is enough to show there exists $s\in R-P$ such that $D(s)\cap V(P)\subseteq U_r$. 

If $M_P=0$ then there exists $s\in R-P$ such that $M_s=0$ and $D(s)\cap V(P)\subseteq U_r$. Assume that $M_P\not=0$. As $R/P$ is an excellent domain, there exists $s\in R-P$ such that $R_s$ is regular. Let $S=\{Q_1,Q_2,\dots,Q_\ell\}$ be the various primes appearing in a prime filtration of $P^{[p^e]}M_P\subseteq M_P$. The prime $P$ is the unique minimal element of $S$. Without loss of generality let $Q_1=P$. By prime avoidance there exists $t\in Q_2\cap \cdots \cap Q_\ell-P$. Replacing $s$ with $st$, we may assume for each $Q\in D(s)\cap V(P)$ that $R_Q/PR_Q$ is regular and $PR_Q$ is the only prime appearing in a prime filtration of $P^{[p^e]}M_Q\subseteq M_Q$. We claim that $D(s)\cap V(P)\subseteq U_r$.

Let $Q\in D(s)\cap V(P)$ and let $\underline{x}=x_1,\dots,x_d$ be a regular system of parameters for $R_Q/PR_Q$. In particular, $QR_Q=(P,\underline{x})$. As $R_Q/PR_Q$ is the only prime factor appearing in a prime filtration of $P^{[p^e]}M_Q\subseteq M_Q$, it is easy to see that the length of the filtration is precisely $\lambda_{R_P}(M_P/P^{[p^e]}M_P)$. It readily follows that $$\lambda_{R_Q}(M_Q/Q^{[p^e]}M_Q)\leq \lambda_{R_P}(M_P/P^{[p^e]}M_P)\lambda_{R_Q}(R_Q/(P,(\underline{x})^{[p^e]})R_Q).$$ By Theorem~\ref{Kunz's Theorem}, $$\lambda_{R_Q}(R_Q/(P,(\underline{x})^{[p^e]})R_Q)=p^{e\dim(R_Q/PR_Q)},$$ which is equal to $p^{e(\height(Q)-\height(P))}$ as $R$ is locally equidimensional. Dividing the above inequality by $p^{e\height(Q)}$ shows $\lambda_e(Q)\leq \lambda_e(P)$, hence $Q\in U_r$.\end{proof}

\begin{Remark}\label{Dense upper semi-continuity remark} The proof of Theorem~\ref{Shepherd-Barron's result for modules} shows that the function $\lambda_e$ is dense upper semi-continuous. That is for each $P\in\Spec(R)$ there exists a dense open set $U$ containing $P$ such that $\lambda_e(Q)\leq \lambda_e(P)$ for each $Q\in U$.
\end{Remark}

\subsection{F-signature}\label{F-signature} Suppose that $(R,\m,k)$ is an F-finite local ring of dimension $d$, and prime characteristic $p$.  For each $e\in\N$, let $a_e(R)$ be the maximal number of free $R$-summands appearing in various direct sum decompositions of $F^e_*R$, and call $a_e(R)$ the $e$th Frobenius splitting number of $R$. Suppose that $F^e_*R\cong R^{\oplus n}\oplus M_e$ where $M_e$ does not contain a free summand. Consider the sets $I_e=\{r\in R\mid \varphi(F^e_*r)\in\m, \forall \varphi\in\Hom_R(F^e_*R,R)\}$, introduced by Aberbach and Enescu \cite{AberbachEnescu2005}. Then, one can easily verify that $I_e$ is an ideal of $R$, and that $F^e_*I_e\cong\m R^{\oplus n}\oplus M_e$. Therefore $a_e(R)$ is the maximal number of free $R$-summands appearing in any direct sum decomposition of $F^e_*R$. Tucker proved the following in \cite{Tucker2012}.

\begin{Theorem}[{\cite[Main Result]{Tucker2012}\label{F-signature exists local}}] Let $(R,\m,k)$ be a local F-finite ring of prime characteristic $p$. Then the following limit exists, $$\lim_{e\rightarrow\infty}\frac{a_e(R)}{p^{e\gamma(R)}}.$$ Its limit is denoted $\s(R)$, and is called the F-signature of $R$. Moreover, $a_e(R)=\s(R)p^{\gamma(R)}+O(p^{e(\gamma(R)-1)})$.
\end{Theorem}

\begin{Remark} The fact that $a_e(R)=\s(R)p^{\gamma(R)}+O(p^{e(\gamma(R)-1)})$ can be pieced together from results in \cite{Tucker2012} and \cite{BST2012}.  See \cite[Theorem 3.6] {PolstraTucker} for a direct proof.
\end{Remark}

 If $R$ is regular then Theorem~\ref{Kunz's Theorem} implies $a_e(R)=p^{e\gamma(R)}$ for each $e\in\N$, hence $\s(R)=1$. Huneke and Leuschke proved the reverse implication holds as well.

\begin{Theorem}[{\cite[Corollary 16]{HunekeLeuschke}\label{s=1 local}}] Let $(R,\m,k)$ be an F-finite local of prime characteristic. Then $\s(R)=1$ if and only if $R$ is regular.
\end{Theorem}

Huneke and Leuschke prove that, if $(R,\m,k)$ is Gorenstein, then the positivity of the F-signature is equivalent to $R$ being F-rational, which is equivalent to $R$ being strongly F-regular under the Gorenstein hypothesis by \cite[Corollary 4.7 (a) and Theorem 5.5 (f)]{HHTAMS}. Aberbach and Leuschke extend the equivalence of positivity of the F-signature and strong F-regularity to all local F-finite rings in \cite{AberbachLeuschke}.

\begin{Theorem}[{\cite[Main Result]{AberbachLeuschke}\label{s>0 local}}\footnote{See \cite[Theorem 5.1]{PolstraTucker} for a simpler proof.}] Let $(R,\m,k)$ be a local ring of prime characteristic and F-finite. Then $\s(R)>0$ if and only if $R$ is strongly F-regular.
\end{Theorem} 

Any strongly F-regular local ring is a domain. So the study of the F-signature is typically of interest when $R$ is a domain.

\begin{Remark}\label{F-signature of a Module} Define the F-signature of a finitely generated $R$-module $M$ as follows. Let $a_e(M)$ be the largest rank of a free module appearing in various or, equivalently, in a single direct sum decomposition of $F^e_*M$. The F-signature of $M$ is defined to be $\s(M)=\lim_{e\rightarrow \infty}a_e(M)/p^{e\gamma(R)}$, which exists by some simple reductions to the scenario that $M=R$. Moreover, positivity of $\s(M)$ implies positivity of $\s(R)$. To see this one only needs to observe $a_e(M)\leq a_e(R^{\oplus \mu(M)})=\mu(M)a_e(R)$. It is also the case $\s(M)=\rank_R(M)\s(R)$, see \cite[Theorem 4.11]{Tucker2012}.
\end{Remark}

The third author naturally extends the notion of F-signature to all local rings which are not assumed to be F-finite in \cite{Yao2006}. Let $E_R(k)$ denote the injective hull of $k$ and let $u\in E_R(k)$ generate the socle. Let $I_e=\{r\in R\mid u\otimes F^e_*r=0\in E_R(k)\otimes F^e_*R\}$. If $(R,\m,k)$ is F-finite and of dimension $d$, then $\lambda(R/I_e)/p^{ed}=a_e(R)/p^{e\gamma(R)}$. If $(R,\m,k)$ is not necessarily F-finite, the F-signature of $R$ is defined to be $\lim_{e\rightarrow \infty}\lambda(R/I_e)/p^{ed}$. The third author's observations in \cite{Yao2006} and Tucker's work in \cite{Tucker2012} provide the existence of the F-signature of a non-F-finite local ring. Moreover, Theorem~\ref{s=1 local} and Theorem~\ref{s>0 local} remain valid without the F-finite assumption. 

The third author has shown that if $(R,\m)$ is a non-regular local of of prime characteristic $p$ and dimension $d$, then $\s(R)< 1-\frac{1}{d!p^d}$, see \cite[Theorem 3.1]{Yao2006}. We discuss how Theorem~\ref{HK really small means regular local} provides the existence of a constant $\delta>0$, depending only on the dimension of a local ring, such that if $(R,\m,k)$ is local of dimension $d$, of any prime characteristic, and non-regular, then $\s(R)<1-\delta$. First we recall the following.

\begin{Proposition}[{\cite[Proposition 14]{HunekeLeuschke}}]\label{HL Upper bound of F-signature} Let $(R,\m)$ be an F-finite strongly F-regular local ring. Let $e(R)$ be the Hilbert-Samuel multiplicity of $R$. Then $(e(R)-1)(1-\s(R))\geq \e(R)-1$. In particular, if $R$ is non-regular then $\s(R)\leq 1-\frac{\e(R)-1}{e(R)-1}$.
\end{Proposition}
\begin{Remark}\label{Remark for HL Upper bound} Let $(R,\m,k)$ be a local ring of characteristic $p$ and dimension $d$. As $\m^{[p^e]}\subseteq \m^{p^e}$, $\lambda(R/\m^{[p^e]})\geq \lambda(R/\m^{p^e})$. Dividing by $p^{ed}$ and letting $e\rightarrow \infty$ shows $\e(R)\geq  \frac{e(R)}{d!}$. Hence if $R$ is F-finite, strongly F-regular, but non-regular, $\s(R)\leq 1-\frac{\e(R)-1}{e(R)-1}\leq 1-\frac{\e(R)-1}{d!\e(R)-1}$.
\end{Remark}

\begin{Theorem}\label{F-signature big means regular} Fix $d\in\N$. There is a number $\delta>0$ such that if $(R,\m,k)$ is an F-finite of dimension $d$, of any prime characteristic, and such that $\s(R)\geq 1-\delta$, then $R$ is regular.
\end{Theorem}
\begin{proof} Assume that $R$ is non-regular and let $\delta$ be as in Theorem~\ref{HK really small means regular local}. If $R$ is not strongly F-regular, then $\s(R)=0$ by Theorem~\ref{s>0 local}. Thus we may assume $R$ is strongly F-regular, in particular $R$ is a domain and $\e(R)>1+\delta$ by Theorem~\ref{HK really small means regular local}. By Proposition~\ref{HL Upper bound of F-signature} and Remark~\ref{Remark for HL Upper bound},
\begin{align*}
\s(R)&\leq 1-\frac{\e(R)-1}{d!\e(R)-1}=\frac{d!\e(R)-\e(R)}{d!\e(R)-1}\\
&=\frac{d!-1}{d!-\frac{1}{\e(R)}}<\frac{d!-1}{d!-\frac{1}{\delta+1}}=1-\frac{\delta}{d!(\delta +1)-1}.
\qedhere
\end{align*}
\end{proof}

Given a local ring $(R,\m,k)$ of prime characteristic $p$ and of dimension $d$, let $s_e(R)=\lambda(R/I_e)/p^{ed}$ and call this number the $e$th normalized Frobenius splitting number of $R$. In \cite{Yao2006}, the third author proves if $(R,\m)\rightarrow (S,\n)$ is flat, then for each $e\in\N$, $s_e(R)\geq s_e(T)$. In other words, the normalized Frobenius splitting numbers can only decrease after flat extensions. We formally state this theorem for future reference.

\begin{Theorem}[{\cite[Theorem 5.4 (3), Theorem 5.5]{Yao2006}}]\label{F-signature comparison} Let $(R,\m,k)\rightarrow (T,\n,l)$ be a flat local ring homomorphism of local rings of prime characteristic $p$ and let $M$ be a finitely generated $R$-module. Then $s_e(M)\geq s_e(M\otimes_RT)$ for each $e\in\N$, hence $\s(M)\geq \s(M\otimes_R T)$. Moreover, if $T/\m T$ is regular, then $s_e(M)=s_e(M\otimes_RT)$ for each $e\in \N$, hence $\s(M)=\s(M\otimes_R T)$.
\end{Theorem}

\subsection{Cartier subalgebras and the F-signature}\label{Cartier subalgebras}   In \cite{BST2012, BST2013} Blickle, Schwede, and Tucker use the language of Cartier subalgebras to greatly generalize the notion of the F-signature.  Their generalization of the F-signature provide the correct framework to answer a question of Aberbach and Enescu, see \cite[Question 4.9]{AberbachEnescu2005} and \cite[Remark 4.6]{BST2012}.

We make the assumption that $R$ is an F-finite ring, not necessarily local. One can make $\C^R=\oplus_{e\in\N}\Hom_R(F^e_*R,R)$ a graded $\F_p$-algebra in a natural way. The $0$th graded piece of $\C^R$ is $\Hom_R(R,R)\cong R$. If $\varphi\in \Hom_R(F^e_*R,R)$ and $\psi\in \Hom_R(F^{e'}_*R,R)$, then we let $\varphi\bullet \psi= \varphi\circ F^e_*\psi\in \Hom_R(F^{e+e'}_*R,R)$. One should observe that $\C^R$ is non-commutative and that $R\cong\Hom_R(R,R)$ is not central in $\C^R$. If $r\in R$, $\varphi\in\Hom_R(F^e_*R,R)$, and $F^e_*s\in F^e_*R$, then $r\bullet \varphi(F^e_*s)=r\varphi(F^e_*s)=\varphi(r F^e_*s)=\varphi(F^e_*r^{p^e}s)\not=\varphi(F^e_*rs)=\varphi(F^e_*s)\bullet r$.

A Cartier subalgebra $\D$ is a graded $\F_p$-subalgebra of $\C^R$ such that the $0$th graded piece of $\D$ is $\Hom_R(R,R)$, which is all of the $0$th graded piece of $\C^R$. Let $\D_e$ denote the $e$th graded piece of $\D$. We refer the reader to \cite{pinversemapssurvey} for a more thorough introduction to Cartier subalgebras.

 Given a Cartier subalgebra $\D$ we call a summand $M$ of $F^e_*R$ a $\D$-summand if $M\cong R^{\oplus n}$ is free and the map $F^e_*R\rightarrow M\cong R^{\oplus n}$ is a direct sum of elements of $\D_e$. The assumption that $\D_0=\Hom_R(R,R)$ implies that the chosen isomorphism of $M\cong R^{\oplus n}$ does not affect whether $M$ is a $\D$-summand or not. If $R=(R,\m,k)$ is local, then the $e$th Frobenius splitting number of $(R,\D)$ is defined to be the maximal rank of a free $\D$-summand appearing in various direct sum decompositions of $F^e_*R$ and is denoted $a_e(R,\D)$. As with the usual Frobenius splitting numbers, one only needs to look at a single direct sum decomposition of $F^e_*R$ to determine $a_e(R,\D)$, see \cite[Proposition 3.5]{BST2012}. Observe that if $\D=\C^R$ then $a_e(R,\D)=a_e(R)$ is the usual $e$th Frobenius splitting number of $R$. To ease notation, we will typically write $\s_e(R,\D)$ to represent $a_e(R,\D)/p^{e\gamma(R)}$.
 
Suppose $R$ is an F-finite domain. We define two classes of Cartier subalgebras which arise from geometric considerations, see \cite{HaraYoshida, HaraWatanabe, Takagi}.  Let $0 \neq \a \subseteq R$ be an ideal. For $t \in \R_{\geq 0}$, define
$\C^{\a^t} = \bigoplus_{e \geq 0} \C^{\a^t}_e$, where
\[
\begin{array}{rcl}
\C^{\a^t}_e &=& F^e_*\a^{\lceil t(p^e - 1)\rceil}\Hom_R(F^e_*R,R) \\
&=& \{ \phi ( F^e_*x \cdot \blank) \mid F^e_*x \in F^e_*\a^{\lceil t(p^e - 1)\rceil } \mbox{ and } \phi \in \Hom_R(F^e_*R,R) \}.
\end{array}
\]
Suppose $R$ is an F-finite normal domain and $\Delta$ is an effective $\Q$-divisor on $\Spec(R)$, define
$\C^{(R,\Delta)} = \bigoplus_{e \geq 0} \C^{(R,\Delta)}_e$, where
\[
\begin{array}{rcl}
\C^{(R,\Delta)}_e &=& \{ \phi \in \Hom_R(F^e_*R,R) \mid \Delta_\phi \geq \Delta\} \\ &=& \im\left(  \Hom_R(F^e_*R(\lceil (p^e - 1) \Delta \rceil),R) \to \Hom_R(F^e_*R,R)\right).
\end{array}
\]

 Given a Cartier subalgebra $\D$, let $\Gamma_\D=\{e\in\N\mid \D_e\not=0\}$. One can easily check that $\Gamma_\D$ is a subsemigroup of $\N$. Blickle, Schwede, and Tucker prove the following.

\begin{Theorem}[{\cite[Theorem 3.11]{BST2012}}]\label{F-sig pairs} Let $(R,\m,k)$ be an F-finite local domain and let $\D$ be a Cartier subalgebra. Then the following limit exists, $$\displaystyle\lim_{e\in\Gamma_\D\rightarrow \infty}\s_e(R,\D).$$ Its limit is denoted $\s(R,\D)$ and is called the F-signature of $(R,\D)$.
\end{Theorem}

A Cartier subalgebra $\D$, or the pair $(R,\D)$, is called strongly F-regular if for every $r\in R$ there is an $e\in\Gamma_\D$ and $\varphi\in\D_e$ such that $\varphi(F^e_*r)=1$. Blickle, Schwede, and Tucker improve Theorem~\ref{s>0 local} by showing the equivalence between positivity of $\s(R,\D)$ and strong F-regularity of $(R,\D)$.

\begin{Theorem}[{\cite[Theorem 3.18]{BST2012}}]\label{s>0 Cartier local} Let $(R,\m,k)$ be an F-finite domain and $\D$ a Cartier subalgebra. Then $\s(R,\D)>0$ if and only if $(R,\D)$ is strongly F-regular.
\end{Theorem}

\subsection{Basic element results}\label{Basic Element Results} Unless otherwise stated, the results which we recall in this subsection are characteristic independent. We only require that $R$ is Noetherian of finite Krull dimension $d$.  The first result we recall is a weakening of the Forster-Swan Theorem. We refer the reader to \cite{Forster1964} and \cite{Swan} for the original statements. We also recommend reading \cite{EisenbudEvans} and the material surrounding \cite[Theorem 5.8]{Matsumura}  for a historical discussion of the Theorem.

 \begin{Theorem}[Forster-Swan Theorem]
 \label{Forster-Swan} 
 If $A$ is a Noetherian ring of finite Krull dimension, e.g., $A$ is of prime characteristic and F-finite, and $N$ is a finitely generated $R$-module, then $\mu_A(N)\leq \max\{\mu_{A_P}(N_P)\mid P\in \Supp(N)\}+\dim(A)$. 
 \end{Theorem}
 
Another result which can be obtained by basic element techniques is Serre's Splitting Theorem.

\begin{Theorem}[{\cite[Theorem 1]{Serre}}]\label{Serre's Splitting Theorem} Let $R$ be a Noetherian ring of dimension $d$. If $\Omega$ is a projective module, locally free of rank at least $d+1$ at each prime ideal, then $\Omega$ contains a free summand.
\end{Theorem}

Stafford greatly generalized Serre's Splitting Theorem to all finitely generated modules.
 
 \begin{Theorem}[{\cite{Stafford}\label{Generalized Serre's Theorem}\footnote{The authors of this paper have recently written a paper providing alternative proofs of Theorem~\ref{Generalized Serre's Theorem} in the commutative case, see \cite{DSPY}. Moreover, the results of \cite{DSPY} allows us to establish the existence of a global F-signature of Cartier subalgebra in Theorem~\ref{Global F-signature for Cartier algebras}. }}] Let $R$ be a Noetherian ring of finite Krull dimension $d$, e.g., $R$ is an F-finite ring of prime characteristic $p$. Suppose that $M$ is a finitely generated $R$-module and that for each $P\in\Spec(R)$, $M_P$ contains a free $R_P$-summand of rank at least $d+1$, then $M$ contains a free summand. 
 \end{Theorem}

We remark that Stafford's results in \cite{Stafford} is a great generalization of results of Eisenbud and Evans in \cite{EisenbudEvans}. For example, Stafford establishes Theorem~\ref{Generalized Serre's Theorem} in the scenario that $R$ is non-commutative. Theorem~\ref{Generalized Serre's Theorem} is crucial to establish the existence of the global F-signature in Theorem~\ref{Global F-signature exists}. 

We now introduce some terminology in order to recall another Theorem from \cite{DSPY}. Let $R$ be a commutative Noetherian ring, $M$ a finitely generated $R$-module. Let $\E$ be a submodule of $\Hom_R(M,R)$. We say that a summand $N$ of $M$ is a free $\E$-summand if $N\cong R^{\oplus n}$ is free, and the projection $\varphi: M\rightarrow N\cong R^{\oplus n}$ is a direct sum of  elements of $\E$. Observe that that choice of an isomorphism $N\cong R^{\oplus n}$ does not affect whether or not $N$ is an $\E$-summand. 
 
\begin{Theorem}[{\cite[Theorem C]{DSPY}}]\label{Generalized Serre 2} Let $R$ be a commutative Noetherian ring of dimension $d$, let $M$ a finitely generated $R$-module, and let $\E$ be an $R$-submodule of $\Hom_R(M,R)$. Assume that, for each $P\in \Spec(R)$, $M_P$ contains a free $\E_P$-summand of rank at least $d+1$. Then $M$ contains a free $\E$-summand.
\end{Theorem}

Theorem~\ref{Generalized Serre 2} applies to Cartier algebras. 
The assumption that $\D$ is a Cartier algebra implies that $\D_e\subseteq \Hom_R(F^e_*R,R)$ is an $R$-submodule, and Theorem~\ref{Generalized Serre 2} yields the following.

\begin{Theorem}\label{Local to global splitting numbers Cartier} Let $R$ be an F-finite domain, $\D$ be a Cartier algebra, and $e \in \Gamma_\D$. Suppose that, for all $P\in\Spec(R)$, we have $a_e(R_P,\D_P)\geq \dim(R)+m$, where $m$ is a fixed positive integer. Then $a_e(R,\D)\geq m$.
\end{Theorem}

\subsection{Existence of limits} In this subsection, we reinstate that all rings being considered are of prime characteristic $p$. In \cite{PolstraTucker}, the second author and Tucker develop a unified approach to the local theory of Hilbert-Kunz multiplicity and F-signature. They use the following well-known lemma as a guide to establish the existence of limits in positive characteristic commutative algebra.

\begin{Lemma}[{\cite[Lemma 3.5]{PolstraTucker}\label{Sequence Lemma}}] Let $\{\lambda_e\}_{e \in \N}$ be a sequence of real numbers, $p$ a prime number, and $\gamma \in \N$. Suppose that $\{\frac{1}{p^{e\gamma}}\lambda_e\}$ is a bounded sequence of real numbers.
\begin{enumerate}
\item\label{Sequence Lemma 1} If there exists a positive constant $C \in \R$ such that $\frac{1}{p^{(e+1)\gamma}} \lambda_{e+1} \leq \frac{1}{p^{e\gamma}} \lambda_e + \frac{C}{p^e}$ for all $ e \in \N$, then the limit $\lambda = \lim_{e \to \infty} \frac{1}{p^{e\gamma}} \lambda_e$ exists and $\lambda - \frac{1}{p^{e\gamma}} \lambda_e \leq \frac{2C}{p^e}$ for all $e \in \N$.
\item \label{Sequence Lemma 2} If there exists a positive constant $C \in \R$ such that $\frac{1}{p^{e\gamma}} \lambda_e \leq  \frac{1}{p^{(e+1)\gamma}} \lambda_{e+1} + \frac{C}{p^e}$ for all $e \in \N$, then the limit $\lambda = \lim_{e \to \infty} \frac{1}{p^{e\gamma}} \lambda_e$ exists and $\frac{1}{p^{e\gamma}} \lambda_e - \lambda \leq \frac{2C}{p^e}$ for all $e \in \N$.
\item \label{Sequence Lemma 3} If there exists a positive constant $C \in \R$ such that $| \frac{1}{p^{(e+1)\gamma}} \lambda_{e+1} - \frac{1}{p^{ed}} \lambda_e    | \leq \frac{C}{p^e}$ for all $e \in \N$, then the limit $\lambda = \lim_{e \to \infty} \frac{1}{p^{e\gamma}} \lambda_e$ exists and $|\frac{1}{p^{e\gamma}} \lambda_e - \lambda | \leq \frac{2C}{p^e}$ for all $e \in \N$.  In particular, $\lambda_e=\lambda p^{e\gamma}+O(p^{e(\gamma-1)})$.
\end{enumerate}
\end{Lemma}

Suppose $(R,\m,k)$ is a local F-finite domain. Suppose that $\lambda_e$ is either $\mu(F^e_*R)$ or $a_e(R)$. If $\lambda_e=\mu(F^e_*R)$, then $\e(R)=\lim_{e\rightarrow \infty}\lambda_e/p^{e\gamma(R)}$ and if $\lambda_e=a_e(R)$ then $\s(R)=\lim_{e\rightarrow \infty} \lambda_e/p^{e\gamma(R)}$. We outline how to verify the sequences $\{\mu(F^e_*R)\}$ and $\{a_e(R)\}$ satisfy the hypotheses of (\ref{Sequence Lemma 3}) of Lemma~\ref{Sequence Lemma} under the local hypothesis. We do this for the purpose of pointing out difficulties in the non-local case. We refer the reader to the proof of \cite[Theorem 3.2]{PolstraTucker} for details.

\begin{Lemma}\label{Local Observation 1} Let $(R,\m,k)$ be an F-finite local domain and $M$ a finitely generated $R$-module.  There is a constant $C\in\R$ such that for each $e\in\N$, $\mu(F^e_*M)\leq Cp^{e\gamma(M)}$. In particular, if $M$ is a torsion $R$-module, $\mu(F^e_*M)\leq Cp^{e(\gamma(R)-1)}$.
\end{Lemma}

The observation that there exists $C\in\R$ such that $\mu(F^e_*M)\leq Cp^{e\gamma(M)}$ for each $e\in\N$ is easily reduced to the observation that $\lambda(M/\m^{[p^e]}M)=O(p^{e\dim(M)})$.  It is not clear in the non-local case that there must be a constant $C\in\R$ such that $\mu(F^e_*M)\leq Cp^{e\gamma(M)}$ for each $e\in\N$. We prove such a constant exists in Corollary~\ref{Upper bound on generators}.

The following elementary lemma allows one to verify $\{\mu(F^e_*R)\}$ and $\{a_e(R)\}$ satisfy the remaining conditions in (\ref{Sequence Lemma 3}) of Lemma~\ref{Sequence Lemma}.
\begin{Lemma}[{\cite[Lemma 2.1]{PolstraTucker}}]\label{Local Observation 2} Let $(R,\m,k)$ be a local ring, not necessarily of prime characteristic, and let $M'\rightarrow M\rightarrow M''\rightarrow 0$ be a right exact sequence of finitely generated $R$-modules. Then \begin{enumerate}
\item\label{Local Observation 2 part 1} $\mu(M'')\leq\mu(M)\leq \mu(M')+\mu(M'')$,
\item\label{Local Observation 2 part 2} $\frk(M'')\leq \frk(M)\leq \frk(M')+\mu(M'')$.
\end{enumerate}
\end{Lemma}

Recall that we are trying to verify the sequences $\{\mu(F^e_*R)\}$ and $\{a_e(R)\}$ satisfy the hypotheses of \ref{Sequence Lemma 3} of Lemma~\ref{Sequence Lemma} under the assumption $(R,\m,k)$ is an F-finite local domain of prime characteristic $p$. As $R$ is an F-finite domain, there are short exact sequences $$0\rightarrow R^{\oplus p^{\gamma(R)}}\rightarrow F_*R\rightarrow T\rightarrow 0$$ and $$0\rightarrow F_*R\rightarrow R^{\oplus p^{\gamma(R)}}\rightarrow T'\rightarrow 0$$ so that $T'$ and $T''$ are torsion $R$-modules. As restricting scalars is exact, for each $e\in\N$ there are short exact sequences $$0\rightarrow F^e_*R^{\oplus p^{\gamma(R)}}\rightarrow F^{e+1}_*R\rightarrow F^e_*T\rightarrow 0$$ and $$0\rightarrow F^{e+1}_*R\rightarrow F^e_*R^{\oplus p^{\gamma(R)}}\rightarrow F^e_*T'\rightarrow 0.$$ As $R$ is local, $\mu(F^e_*R^{\oplus p^{\gamma(R)}})=p^{\gamma(R)}\mu(F^e_*R)$ and $\frk(F^e_*R^{\oplus p^{\gamma(R)}})=p^{\gamma(R)}\frk(F^e_*R)=p^{\gamma(R)}a_e(R)$. It readily follows by (\ref{Local Observation 2 part 1}) and (\ref{Local Observation 2 part 2}) of Lemma~\ref{Local Observation 2} that $$\left|\frac{\mu(F^{e+1}_*R)}{p^{(e+1)\gamma(R)}}-\frac{\mu(F^e_*R)}{p^{e\gamma(R)}}\right|\leq \max\left\{\frac{\mu(F^e_*T)}{p^{e\gamma(R)}}, \frac{\mu(F^e_*T')}{p^{e\gamma(R)}}\right\}$$
and $$\left|\frac{a_{e+1}(R)}{p^{(e+1)\gamma(R)}}-\frac{a_e(R)}{p^{e\gamma(R)}}\right|\leq \max\left\{\frac{\mu(F^e_*T)}{p^{e\gamma(R)}}, \frac{\mu(F^e_*T')}{p^{e\gamma(R)}}\right\}.$$ Lemma~\ref{Local Observation 1} provides the existence of a constant $C\in\R$ such that $\max\left\{\frac{\mu(F^e_*T)}{p^{e\gamma(R)}}, \frac{\mu(F^e_*T')}{p^{e\gamma(R)}}\right\}\leq \frac{C}{p^e}$ for each $e\in\N$. Thus the sequences $\{\mu(F^e_*R)\}$ and $\{a_e(R)\}$ satisfy (\ref{Sequence Lemma 3}) of Lemma~\ref{Sequence Lemma}.

Let $R$ be a Noetherian ring, not necessarily local, and $M'\rightarrow M\rightarrow M''\rightarrow 0$ a right exact sequence of finitely generated $R$-modules. Then  (\ref{Local Observation 2 part 1}) of Lemma~\ref{Local Observation 2} remains true without the local hypothesis. It is also clear that $\frk(M'')\leq \frk(M)$. It is not clear whether or not the inequality $\frk(M)\leq \frk(M')+\mu(M'')$ holds without the local hypothesis. However, if $R$ is of finite dimension, we show in Lemma~\ref{Free Rank Lemma} that $\frk(M)\leq \frk(M')+\mu(M'')+\dim(R)$.

\subsection{Uniform convergence and semi-continuity results}\label{Uniform Convergence} In this subsection we recall some results proved by Smirnov in \cite{Smirnov2015}, the second author in \cite{Polstra2015}, and the second author and Tucker in \cite{PolstraTucker} that will be of use in later sections. 

\begin{Proposition}[{\cite[Proposition 3.3]{Polstra2015}}]\label{Uniform Bound of HK} Let $R$ be either F-finite or essentially of finite type over an excellent local ring and $M$ a finitely generated $R$-module. There is a constant $C\in \R$ such that for each $P\in\Spec(R)$ and for each $e\in\N$, $\lambda(M_P/P^{[p^e]}M_P)\leq Cp^{e\dim(M_P)}.$
\end{Proposition}

\begin{Theorem}[{\cite[Theorem 5.1]{Polstra2015}}]\label{Uniform Convergence of HK} Let $R$ be either F-finite or essentially of finite type over an excellent local ring and $M$ a finitely generated $R$-module. The functions $\lambda_e:\Spec(R)\rightarrow \R$ sending $P\mapsto \lambda_{R_P}(M_P/P^{[p^e]}M_P)/p^{e\height(P)}$ converge uniformly to their limit, namely the function $\e:\Spec(R)\rightarrow \R$ sending $P\mapsto \e(M_P)$, the Hilbert-Kunz multiplicity of $M_P$.
\end{Theorem}

Let $R$ and $M$ be as in Theorem~\ref{Uniform Convergence of HK} but assume that $R$ is locally equidimensional. Then the functions $\lambda_e$ are upper semi-continuous by Theorem~\ref{Shepherd-Barron's result for modules}. Thus Smirnov's theorem on the upper semi-continuity of Hilbert-Kunz multiplicity holds for finitely generated modules. See \cite[Main Theorem]{Smirnov2015}.

\begin{Corollary}\label{Upper Semi-Continuity of HK} Let $R$ be a locally equidimensional ring which is either F-finite or essentially of finite type over an excellent local ring and let $M$ be a finitely generated $R$-module. Then the Hilbert-Kunz multiplicity function $\e:\Spec(R)\rightarrow \R$ sending $P\mapsto \e(M_P)$ is upper semi-continuous.
\end{Corollary}

The following theorem is the analogue of Theorem~\ref{Uniform Convergence of HK} for the sequence of normalized Frobenius splitting number functions.

\begin{Theorem}[{\cite[Theorem 5.6]{Polstra2015}}]\label{Uniform Convergence of a_e} Let $R$ be either F-finite or essentially of finite type over an excellent local ring. The functions $\s_e:\Spec(R)\rightarrow \R$ sending $P\mapsto \s_e(R_P)$ converge uniformly to their limit, namely the function $\s:\Spec(R)\rightarrow \R$ sending $P\mapsto \s(R_P)$, the F-signature of $R_P$.
\end{Theorem}

\begin{Remark}\label{Uniform Convergence Remark} We remark that Theorem~\ref{Uniform Convergence of a_e} is easily generalized to all finitely generated $R$-modules. In fact, the proof provided in \cite{Polstra2015} shows this. One only needs to observe that the statements made about the $R$-module $R$ apply to all finitely generated $R$-modules $M$.
\end{Remark}

Let $R$ satisfy the hypotheses of Theorem~\ref{Uniform Convergence of a_e}, $M$ a finitely generated $R$-module, and suppose that $R$ is strongly F-regular. Enescu and the third author show in \cite{EnescuYao} under such hypotheses that the functions $\s_e$ in Theorem~\ref{Uniform Convergence of a_e} are lower semi-continuous.  See \cite[Corollary~2.5, Theorem~5.1, and Remark~5.5]{EnescuYao}\footnote{See Remark~\ref{Assume s>0 remark} for details on why the hypotheses of \cite[Theorem~5.1]{EnescuYao} are satisfied when $R$ is strongly F-regular.}. Hence the following Corollary.

\begin{Corollary}[{\cite[Theorem 5.7]{Polstra2015}}]\label{Lower semi-continuity of the F-signature} Let $R$ be either F-finite or essentially of finite type over an excellent local ring and $M$ a finitely generated $R$-module. The F-signature function $\s:\Spec(R)\rightarrow \R$ sending $P\mapsto \s(M_P)$ is lower semi-continuous.
\end{Corollary}

Let $R$ be an F-finite domain and $\D$ a Cartier subalgebra. It is unknown if the functions $\s^\D_e:\Spec(R)\rightarrow \R$ sending $P\mapsto \s_e(R_P,\D_P)$ converge uniformly to their limit, namely the F-signature function $\s^\D:\Spec(R)\rightarrow \R$ sending $P\mapsto \s(R_P,\D_P)$,  as $e\in\Gamma_\D\rightarrow \infty$. However, there is a condition one could impose on a Cartier subalgebra which guarantees uniform convergence.

 \begin{Condition}\label{Condtion *} A Cartier subalgebra satisfies condition $(*)$ if we require that for each $\varphi \in \D_{e+1}$ that the natural map $i\circ \varphi\in \D_e$, where $i:F^e_*R\rightarrow F^{e+1}_*R$ is the $p$th power map. 
\end{Condition}

\begin{Theorem}[{\cite[Theorem 6.4]{Polstra2015}}]\label{Uniform Convergence of a_e Cartier} Let $R$ be an F-finite domain and $\D$ a Cartier subalgebra. Let $\s^\D_e:\Spec(R)\rightarrow \R$ be the function sending $P\mapsto \s_e(R_P,\D_P)$. Then there is a constant $C\in\R$ such that for all $P\in \Spec(R)$ for all $e,e'\in \N$, $$\s_e(R_P,\D_P)-\s_{e+e'}(R_P,\D_P)<\frac{C}{p^e}.$$ Moreover, if $\D$ satisfies condition $(*)$ then the constant $C$ can be chosen such that for each $P\in\Spec(R)$ and each $e,e'\in\Gamma_\D$,
$$|\s_e(R_P,\D_P)-\s_{e+e'}(R_P,\D_P)|<\frac{C}{p^e}.$$
Therefore if the Cartier algebra satisfies condition $(*)$, the functions $\s_e^\D$ sending $P\mapsto \s_e(R_P,\D_P)$ converge uniformly to their limit, namely the function $\s^\D$ sending $P\mapsto \s(R_P,\D_P)$.
\end{Theorem}

An arbitrary Cartier subalgebra, including those arising from geometric considerations, will not satisfy condition $(*)$. Recently, the second author and Tucker establish the uniform convergence of the functions $\s_e^\D$ when $\D=\C^{\a^t}$ for some ideal $\a$ and $t>0$ or if $R$ is a normal domain and $\D=\C^{(R,\Delta)}$ for some effective $\Q$-divisor $\Delta$.

\begin{Theorem}[{\cite[Theorem 4.12, Theorem 4.13]{PolstraTucker}}]\label{PTUniformConvergence}
Let $R$ be an F-finite domain. Suppose that either $\a \subseteq R$ a non-zero ideal, and $t \in \R_{\geq 0}$ or $R$ is normal and $\Delta$ is an effective $\Q$-divisor. Let $\D$ be either $\C^{\a^t}$ or $\C^{(R,\Delta)}$. Then there is a constant $C\in\R$ such that for all $P\in\Spec(R)$ and for all $e,e'\in\N$, $$\left|\s_e(R_P,\D_P)-\s_{e+e'}(R_P,\D_P)\right|\leq \frac{C}{p^e}.$$ In particular, the functions $\s_e^\D$ sending $P\mapsto \s_e(R_P,\D_P)$ converge uniformly to their limit, namely the function $\s^\D$ sending $P\mapsto \s(R_P,\D_P)$.
\end{Theorem}

Following the methods of Enescu and the third author in \cite{EnescuYao}, the functions $\s_e^\D$ are easily verified to be lower semi-continuous. Thus a Corollary of Theorem~\ref{Uniform Convergence of a_e Cartier} is the lower semi-continuity of the F-signature of a pair $(R,\D)$.

\begin{Theorem}[{\cite[Theorem 6.4]{Polstra2015}}]\label{Lower semi-continuity of the F-signature Cartier} Let $R$ be an F-finite ring and $\D$ a Cartier subalgebra. The F-signature function $\s^\D:\Spec(R)\rightarrow \R $ sending $P\mapsto \s(R_P,\D_P)$ is lower semi-continuous. 
\end{Theorem}

\section{Global Hilbert-Kunz multiplicity}\label{Global HK} Recall that if $R$ is an F-finite ring then we let $\gamma(R)=\max\{\alpha(Q)\mid Q\in\Min(R)\}$. If $R$ is F-finite and $M$ is a finitely generated $R$-module, define $$\e(M)=\displaystyle\lim_{e\rightarrow \infty}\frac{\mu(F^e_*M)}{p^{e\gamma(R)}}.$$ Theorem~\ref{HK exists} shows that the limit exists, and we call it the \emph{global Hilbert-Kunz multiplicity of $M$}. We observe that $\e(M)$ agrees with the usual Hilbert-Kunz multiplicity of $M$ if $R$ is local, see Remark~\ref{Remark Same as local}. 

The following Lemma is a global version of an observation made by Dutta in \cite{Dutta1983}. Lemma~\ref{Dutta's Lemma} has shown itself to be useful in positive characteristic commutative algebra. Huneke's survey paper \cite{Huneke2013} uses a local version of Lemma~\ref{Dutta's Lemma} to prove the existence of Hilbert-Kunz multiplicity and the F-signature. Lemma~\ref{Dutta's Lemma} is used by the second author in \cite{Polstra2015} to establish the presence of strong uniform bounds found in all F-finite rings.

\begin{Lemma}[{\cite[Lemma 2.2]{Polstra2015}}]\label{Dutta's Lemma}Let $R$ be an F-finite domain. Then there exists a finite set of nonzero primes $\mathcal{S}(R)$, and a constant $C$, such that for every $e\in\N$,
\begin{enumerate}
\item there is a containment of $R$-modules $R^{\oplus p^{e\gamma(R)}}\subseteq F^e_*R$,
\item which has a prime filtration with prime factors isomorphic to $R/Q$, where $Q\in\mathcal{S}(R)$,
\item and for each $Q\in \mathcal{S}(R)$, the prime factor $R/Q$ appears no more than $Cp^{e\gamma(R)}$ times in the chosen prime filtration of $R^{\oplus p^{e\gamma(R)}}\subseteq F^e_*R$.
\end{enumerate}
\end{Lemma}

\begin{Corollary}\label{Upper bound on generators} Let $R$ be an F-finite ring and $M$ a finitely generated $R$-module. Then $\mu(F^e_*M)= O(p^{e\gamma(M)})$.
\end{Corollary}

\begin{proof} Counting minimal number of generators is sub-additive on short exact sequences and restricting scalars is exact. Thus by considering a prime filtration of $M$, we are reduced to showing that if $M=R$ is an F-finite domain, then there is a constant $C$ such that for every $e>0$, $\mu(F^e_* R)\leq Cp^{e\gamma(R)}$.  

Suppose that $R$ is an F-finite domain and let $\mathcal{S}(R)$ and $C$ be as in Lemma~\ref{Dutta's Lemma}. If $\mathcal{S}(R)$ is empty, i.e., for each $e$ we can take the inclusions $R^{\oplus p^{e\gamma(R)}}\subseteq F^e_* R$ to be the surjective as well, then there is nothing to show. For each $e>0$ let $T_e=F^e_*R/R^{\oplus p^{e\gamma(R)}}$. Then we can find a prime filtration of $T_e$, whose prime factors are isomorphic to $R/Q$, where $Q\in\mathcal{S}(R)$, and such a prime factor appears no more than $Cp^{e\gamma(R)}$ times. In particular, $T_e$ has a prime filtration with no more than $C|\mathcal{S}(R)|p^{e\gamma(R)}$ prime factors. By considering the short exact sequence $0\rightarrow R^{\oplus p^{e\gamma(R)}}\rightarrow F^e_*R\rightarrow T_e\rightarrow 0$ and the prime filtration of $T_e$, we have that $\mu(F^e_* R)\leq \mu(R^{\oplus p^{e\gamma(R)}})+\mu(T_e)=p^{e\gamma(R)}+\mu(T_e)\leq (1+C|\mathcal{S}(R)|)p^{e\gamma(R)}$. \end{proof}

\begin{Remark} The reader should observe that Proposition~\ref{Uniform Bound of HK} and Theorem~\ref{Forster-Swan} provide an alternative proof to Corollary~\ref{Upper bound on generators} without directly using Lemma~\ref{Dutta's Lemma}. However, the proof of Proposition~\ref{Uniform Bound of HK} given in \cite{Polstra2015} relies on Lemma~\ref{Dutta's Lemma}.
\end{Remark}
 
 \begin{Lemma}\label{HK Lemma} Let $R$ be an F-finite ring and let $M,N$ be finitely generated $R$-modules which are isomorphic at minimal primes of $R$. Then $\mu(F^e_*M)=\mu(F^e_*N)+O(p^{e(\gamma(R)-1)}).$
 \end{Lemma}
 
 \begin{proof} As $M$ and $N$ are assumed to be isomorphic at minimal primes, we can find right exact sequences $M\rightarrow N \rightarrow T_1\rightarrow 0$ and $N\rightarrow M \rightarrow T_2\rightarrow 0$ such that $T_1$ and $T_2$ are not supported at any minimal prime of $R$. Hence for each $e\in\N$ we have right exact sequences $F^e_*M\rightarrow F^e_*N \rightarrow F^e_*T_1\rightarrow 0$ and $F^e_*N\rightarrow F^e_*M \rightarrow F^e_*T_2\rightarrow 0$. It follows that $$|\mu(F^e_*M)-\mu(F^e_*N)|\leq \max\{\mu(F^e_*T_1), \mu(F^e_*T_2)\}.$$ Therefore by Corollary~\ref{Upper bound on generators}, $$|\mu(F^e_*M)-\mu(F^e_*N)|= O\left(p^{e(\max\{\gamma(T_1),\gamma(T_2)\})}\right).$$ For $i=1,2$, let $I_i=\Ann_R(T_i)$. Then there is a $P_i\in\Spec(R)$ such that $\gamma(T_i)=\alpha(P_i/I_i)+\height(P_i/I_i)$. Observe that $\alpha(P_i/I_i)+\height(P_i/I_i)=\alpha(P_i)+\height(P_i/I_i)<\alpha(P_i)+\height(P_i)\leq \gamma(R)$, which completes the proof of the Lemma.  \end{proof}
 
\begin{Remark}
\label{Remark about HK Lemma}
The method of Lemma~\ref{HK Lemma} shows something a bit stronger. If we set $\Assh(R)=\{P\in\Min(R)\mid \alpha(P)=\gamma(R)\}$ and assume that $M,N$ are finitely generated $R$-modules which are isomorphic at the minimal primes in $\Assh(R)$, then $\mu(F^e_*M)=\mu(F^e_*N)+O(p^{e(\gamma(R)-1)})$. Recall that in Section~\ref{Background} we used $\Assh(R)$ to denote the set of minimal primes $Q$ of a local ring $(R,\m,k)$ such that $\dim(R/Q)=\dim(R)$. The following Lemma justifies our use of $\Assh(R)$.
\end{Remark}

\begin{Lemma}\label{Assh Lemma} Let $(R,\m,k)$ be an F-finite local ring and let $P$ be a minimal prime of $R$. Then $\alpha(P)=\gamma(R)$ if and only if $\dim(R/P)=\dim(R)$.
\end{Lemma}
 
\begin{proof} Observe that if $P\in \Min (R)$, then $\alpha(P)=\alpha(P/P)$ in the local domain $R/P$. By Lemma~\ref{Kunz's Lemma}, $\alpha(P)=\dim(R/P)+\alpha(\m/P)=\dim(R/P)+\alpha(\m)$. This shows a minimal prime $P\in\Min(R)$ in a local ring satisfies $\dim(R)=\dim(R/P)$ if and only if $\alpha(P)=\gamma(R)$. \end{proof}
 
 The following is a Corollary to Lemma~\ref{HK Lemma}.
 
\begin{Corollary}\label{Corollary to HK Lemma} Let $R$ be an F-finite ring and $0\rightarrow M'\rightarrow M\rightarrow M''\rightarrow 0$ a short exact sequence of finitely generated $R$-modules. Then $\mu(F^e_*M)=\mu(F^e_*M'\oplus F^e_*M'')+O(p^{e(\gamma(R)-1)}).$
\end{Corollary}

\begin{proof} First suppose that $R$ is a reduced ring. Then $M$ is isomorphic to $M'\oplus M''$ at minimal primes of $R$. Hence by Lemma~\ref{HK Lemma}, $\mu(F^e_*M)=\mu(F^e_*M'\oplus F^e_*M'')+ O(p^{e(\gamma(R)-1)})$. 
 
 Now suppose that $R$ is not necessarily reduced and choose $e_0$ large enough such that $\sqrt{0}^{[p^{e_0}]}=0$. Denote by $S$ the image of $R$ under the $e_0$th iterate of Frobenius, $F^{e_0}:R\rightarrow R$. Then $S$ is a reduced ring which $R$ is module finite over. Suppose that $N$ is a finitely generated $R$-module. Observe that the elements $n_1,\dots,n_\ell$ are a generating set for $N$ as an $S$-module if and only if $F^{e_0}_*n_1,\dots,F^{e_0}_*n_\ell$ are a generating set for $F^{e_0}_*N$ as an $R$-module. It follows that for each $e\in \N$ that $\mu_R(F^{e+e_0}_*N)=\mu_S(F^e_*N)$, which reduces the proof of the Corollary to the reduced case.  \end{proof}

 One should observe that $\mu(F^e_*M'\oplus F^e_*M'')\leq \mu(F^e_*M')+ \mu(F^e_*M'')$, but equality does not necessarily hold since $R$ is not assumed to be local. If fact, one can not even hope to prove $\mu(F^e_*M)=\mu(F^e_*M')+ \mu(F^e_*M'')+O(p^{e(\gamma(R)-1)})$. If such an inequality held, then one could establish that global Hilbert-Kunz multiplicity was additive on short exact sequences. 
 
\begin{Example}
\label{Global HK not additive}
Global Hilbert-Kunz multiplicity is not additive on direct summands, hence not additive on short exact sequences. Let $R=\F_p\times \F_p$. For each $e\in\N$, $F^e_*R\cong R$, hence $\mu(F^e_*R)=1$ for each $e\in\N$ and $\e(R)=1$. Let $M_1=\F_p\times 0$ and $M_2=0\times \F_p$, the two direct summands of $\F_p$ of $R$. Then for each $e\in\N$, $F^e_*M_1\cong M_1$ and $F^e_*M_2\cong M_2$. Hence $\e(M_1)=1$ and $\e(M_2)=1$, but $\e(M_1\oplus M_2)\not= 2$. Nevertheless, Corollary~\ref{associativity} below shows that global Hilbert-Kunz multiplicity is additive if $R$ is assumed to be a domain.
\end{Example}
 
We now prove the existence of global Hilbert-Kunz multiplicity.
 
 \begin{Theorem}\label{HK exists} Let $R$ be an F-finite ring and $M$ a finitely generated $R$-module. Then the limit $\displaystyle \e(M)=\lim_{e\rightarrow \infty}\mu(F^e_* M)/p^{e\gamma(R)}$ exists. Moreover, there is a constant $C\in \R$ such that for each $e\in \N$,  $\displaystyle \e(M)\leq \frac{\mu(F^e_*M)}{p^{e\gamma(R)}}+\frac{C}{p^e}.$
\end{Theorem}

\begin{proof} Suppose that $0=M_0\subseteq M_1\subseteq \cdots \subseteq M_\ell=M$ is a prime filtration of $M$ with $M_i/M_{i-1}\cong R/Q_i$. Repeated use of Corollary~\ref{Corollary to HK Lemma} allows us to reduce proving the existence of global Hilbert-Kunz multiplicity to the scenario that $M\cong R/Q_1\oplus R/Q_2\oplus \cdots \oplus R/Q_\ell$, a direct sum of modules of the form $R/Q_i$ where $Q_i\in\Spec(R)$. 

Suppose that $\Assh(R)$ is as in Remark~\ref{Remark about HK Lemma}. By rearranging and relabeling as necessary, we may assume that $Q_1,\dots,Q_i\in \Assh(R)$ and $Q_{i+1},\dots,Q_{\ell}\not\in\Assh(R)$. Hence $M$ and $R/Q_1\oplus R/Q_2\oplus \cdots \oplus R/Q_i$ are isomorphic when localized at each $Q\in \Assh(R)$. Thus by Remark~\ref{Remark about HK Lemma}, we are further reduced to the scenario $M\cong R/Q_1\oplus R/Q_2\oplus \cdots \oplus R/Q_\ell$ where each $Q_i\in \Assh(R)$.

A prime $Q$ is an element of $\Assh(R)$ if and only if $F_*(R/Q)$ has rank $p^{\gamma(R)}$ as an $R/Q$-module. It follows that there is a right exact sequence $$M^{\oplus p^{\gamma(R)}}\rightarrow F_*M\rightarrow T\rightarrow 0$$ such that $T_Q=0$ for each $Q\in \Assh(R)$. As restricting scalars is exact, for each $e\in\N$ there is a right exact sequence $$F^e_*M^{\oplus p^{\gamma(R)}}\rightarrow F^{e+1}_*M\rightarrow F^e_*T\rightarrow 0.$$ For each $e\in \N$, $$\mu(F^{e+1}_*M)\leq \mu(F^e_*M^{\oplus p^{\gamma(R)}})+\mu(F^e_*T)\leq p^{\gamma(R)}\mu(F^e_*M) +\mu(F^e_*T).$$ As $T$ is not supported at any prime of $\Assh(R)$, by Corollary~\ref{Upper bound on generators} there is a constant $C\in\R$ such that for each $e\in\N$, after dividing by $p^{(e+1)\gamma(R)}$,$$\frac{\mu(F^{e+1}_*M)}{p^{(e+1)\gamma(R)}}\leq \frac{\mu(F^e_*M)}{p^{e\gamma(R)}} + \frac{C}{p^e}.$$ The theorem follows from (\ref{Sequence Lemma 1}) of Lemma~\ref{Sequence Lemma}.
\end{proof}

 \begin{Corollary}[Associativity Formula]\label{Associativity Formula} Let $R$ be an F-finite ring and $M$ a finitely generated $R$-module. Then $$\mu(F^e_*M)=\mu\left(\bigoplus_{Q\in\Assh(R)}\bigoplus_{i=1}^{\lambda_{R_Q}(M_Q)}F^e_*(R/Q)\right)+O(p^{e(\gamma(R)-1)}).$$ In particular, $$\e(M)=\e\left(\bigoplus_{Q\in\Assh(R)}\bigoplus_{i=1}^{\lambda(M_Q)}F^e_*(R/Q)\right).$$
 \end{Corollary}
 
 \begin{proof} In the proof of Theorem~\ref{HK exists} it was observed that $$\mu(F^e_*M)=\mu(F^e_*(R/Q_1\oplus R/Q_2\cdots \oplus R/Q_\ell))+O(p^{e(\gamma(R)-1)}),$$ where $R/Q_i$ are the various prime factors appearing in a prime filtration of $M$ with $Q_i\in\Assh(R)$. Given a prime $Q\in\Assh(R)$, the number of times $R/Q$ appears as a prime factor in a prime filtration of $M$ is precisely $\lambda_{R_Q}(M_Q)$. \end{proof}

If $(R,\m,k)$ is a local F-finite ring and $M$ a finitely generated $R$-module, then Monsky's original proof of the existence of Hilbert-Kunz multiplicity in \cite{Monsky1983} showed $\mu(F^e_*M)=\e(M)p^{e\gamma(R)}+O(p^{e(\gamma(R)-1)})$. Equivalently, there is a constant $C\in\R$ such that $|\mu(F^e_*M)-\e(M)p^{e\gamma(R)}|\leq C p^{e(\gamma(R)-1)}$. To extend Monsky's original result to the global case, we first record the following application of Theorem~\ref{Forster-Swan}.

\begin{Lemma}
\label{Applicaiton of Forster-Swan} Let $R$ be a Noetherian ring of finite Krull dimension. Suppose that $M$ is a finitely generated $R$-module. Then for each $n\in \N$, $$|\mu(M^{\oplus n})-n\mu(M)|\leq n\dim(R).$$
\end{Lemma} 

\begin{proof} It is easy to see that $\mu(M^{\oplus n})\leq n \mu(M)$. By Theorem~\ref{Forster-Swan} there is a $P\in\Spec(R)$ such that $\mu(M)\leq \mu(M_P)+\dim(R)$. Hence $n\mu_R(M)\leq n\mu_{R_P}(M_P)+n\dim(R)=\mu_{R_P}(M_P^{\oplus n})+n\dim(R)\leq \mu_R(M^{\oplus n})+n\dim(R)$.
\end{proof}

\begin{Theorem}
\label{Global HK Growth Rate}
Let $R$ be an F-finite ring and $M$ a finitely generated $R$-module. Then $\mu(F^e_*M)=\e(M)p^{e\gamma(R)}+O(p^{e(\gamma(R)-1)})$.
\end{Theorem}

\begin{proof} As in the proof of Theorem~\ref{HK exists}, one can reduce all considerations to the scenario $M\cong R/Q_1\oplus \cdots R/Q_\ell$ where each $Q_i\in\Assh(R)$. Hence there is a right exact sequence of the form $$F_*M\rightarrow M^{\oplus p^{\gamma(R)}}\rightarrow T\rightarrow 0$$ such that $T_Q=0$ for each $Q\in\Assh(R)$. For each $e\in\N$, $$\mu(F^e_*M^{\oplus p^{\gamma(R)}})\leq \mu(F^{e+1}_*M)+\mu(F^e_*T).$$ By Lemma~\ref{Applicaiton of Forster-Swan}, $$p^{\gamma(R)}\mu(F^e_*M)\leq \mu(F^{e+1}_*M)+\mu(F^e_*T)+ p^{\gamma(R)} \dim(R).$$ By Corollary~\ref{Upper bound on generators} there is a constant $C\in\R$ such that for each $e\in\N$, $\mu(F^e_*T)\leq Cp^{e(\gamma(R)-1)}$. Dividing by $p^{(e+1)\gamma(R)}$ and applying a crude estimate shows $$\frac{\mu(F^e_*M)}{p^{e\gamma(R)}}\leq \frac{\mu(F^{e+1}_*M)}{p^{(e+1)\gamma(R)}}+\frac{C+\dim(R)}{p^{e}}.$$ The theorem follows from Theorem~\ref{HK exists} and (\ref{Sequence Lemma 3}) of Lemma~\ref{Sequence Lemma}.
\end{proof}

\begin{Corollary}
\label{HK Lower bound}
 Let $R$ be an F-finite ring and $M$ a finitely generated $R$-module. Then the limit $\displaystyle \widetilde{\e}(M)=\lim_{e\rightarrow \infty} \mu(F^e_*M)/p^{\gamma(M)}$ exists and $\widetilde{\e}(M)\geq 1$. Moreover, $\mu(F^e_*M)=\widetilde{\e}(M)p^{e\gamma(M)}+O(p^{e(\gamma(M)-1)}).$
\end{Corollary}
\begin{proof} For existence, apply Theorems~\ref{HK exists} and~\ref{Global HK Growth Rate} to the module $M$, but viewed as an $R/\Ann_R(M)$-module. To see that $\widetilde{\e}(M)\geq 1$, one may assume that $\gamma(M)=\gamma(R)$ and show $\e(M)\geq 1$. The assumption $\gamma(M)=\gamma(R)$ is equivalent to $M_Q\not=0$ for some $Q\in\Assh(R)$. Then $\mu_R(F^e_*M)\geq \mu_{R_Q}(F^e_*M_Q)=\lambda_{R_Q}(M_Q/Q^{[p^e]}M_Q)p^{e\alpha(Q)}\geq \lambda_{R_Q}(M_Q)p^{e\gamma(R)}$. Divide by $p^{e\gamma(R)}$ and let $e\rightarrow \infty$ to see $\e(M)\geq \lambda(M_Q)\geq 1$.
\end{proof}

\begin{Remark}\label{Remark Same as local} Let $(R,\m,k)$ be a local F-finite ring and let $M$ be a finitely generated $R$-module.  Then for each $e\in\N$, $\mu(F^e_*M)/p^{e\gamma(R)}=\lambda(M)/\m^{[p^e]M}/\p^{e\dim(R)}, $ see Lemma~\ref{Assh Lemma}.  In particular, the global Hilbert-Kunz multiplicity of $M$ is the same as the usual Hilbert-Kunz multiplicity of $M$, i.e., $$\lim_{e\rightarrow \infty} \frac{\mu(F^e_*M)}{p^{e\gamma(R)}}=\lim_{e\rightarrow \infty}\frac{\lambda(M/\m^{[p^e]}M)}{p^{e\dim(R)}}.$$
 \end{Remark}

Suppose that $R$ is an F-finite domain. Then for each $P\in\Spec(R)$, $\rank_R(F^e_*R)=\rank_{R_P}(F^e_*R_P)$. It follows that $\mu_R(F^e_*R)/\rank_R(F^e_*R)\geq \mu_{R_P}(F^e_*R_P)/\rank_{R_P}(F^e_*R_P)$ and therefore $\e(R)\geq \e(R_P)$. Theorem~\ref{What is HK} below shows that under such hypotheses, $\e(R)=\max\{\e(R_P)\mid P\in \Spec(R)\}$. It will not always be the case that global Hilbert-Kunz multiplicity is an upper bound of $\{\e(R_P)\mid P\in\Spec(R)\}$, see Example~\ref{Global HK and Z_R example} below. To better describe the scenario in all F-finite rings, let 
\[
Z_R=\{P\in\Spec(R)\mid \height(P)+\alpha(P)=\gamma(R)\}.
\] 
Observe that if $R$ is an F-finite domain, then $Z_R=\Spec(R)$. More generally, $P\in Z_R$ if and only if there is some $Q\in\Min(R)$ such that $\gamma(R_Q)=\gamma(R)$ and $Q\subseteq P$ if and only if there is some $Q\in\Assh(R)$ such that $Q\subseteq P$. Therefore $Z_R=\cup_{Q\in\Assh(R)}V(Q)$ is a closed set.

The following theorem is a generalization of Smirnov's theorem that Hilbert-Kunz multiplicity is upper semi-continuous on the spectrum of rings which are locally equidimensional, see Theorem~\ref{Upper Semi-Continuity of HK}.

\begin{Theorem}\label{Generalizing upper semi-continuity of HK} Let $R$ be an F-finite ring and $M$ a finitely generated $R$-modules. For each $e\in\N$ the function $\mu_e:Z_R\rightarrow \R$ sending $P\mapsto \mu_{R_P}(F^e_*M_P)/p^{e\gamma(R_P)}$ is upper semi-continuous. Moreover, the functions $\mu_e$ converge uniformly to their limit, namely $\e:Z_R\rightarrow \R$ sending $P\mapsto \e(M_P)$. In particular, the function $\e:Z_R\rightarrow \R$ is upper semi-continuous and $\sup\{\e(M_P)\mid P\in Z_R\}=\max\{\e(M_P)\mid P\in Z_R\}$.
\end{Theorem}

\begin{proof} For each $e\in\N$, the function $\tilde{\mu}_e:\Spec(R)\rightarrow \N$ sending $P\mapsto \mu_{R_P}(F^e_*M_P)$ is easily seen to be upper semi-continuous on all of $\Spec(R)$. For each $P\in Z_R$, $\gamma(R_P)=\gamma(R)$. Hence $\mu_e$ is upper semi-continuous on $Z_R$. As $\mu_{R_P}(F^e_*M_P)/p^{e\gamma(R_P)}=\lambda_{R_P}(M_P/P^{[p^e]}M_P)/p^{e\heigh(P)}$, the uniform convergence of $\mu_e$ follows from Theorem~\ref{Uniform Convergence of HK}.
\end{proof}

Our next theorem relates the global Hilbert-Kunz multiplicity of an F-finite ring with the Hilbert-Kunz multiplicities of various localizations of $R$. We first need a lemma.

\begin{Lemma}
\label{Lemma for What is HK}
Let $R$ be an F-finite ring. Suppose $M$ is a finitely generated $R$-module such that $\gamma(M)=\gamma(R)$. There exists $e_0\in\N$ such that for all $e\geq e_0$, $\emptyset \neq \{P \in \Spec(R) \mid \mu(F^e_*M) \leq \mu(F^e_*M_P) + \dim(R)\} \subseteq Z_R$.  In particular, $\{P \mid \mu(F^e_*M_P) = \max\{\mu(F^e_*M_Q)\}\} \subseteq Z_R$ for all $e \ge e_0$.
\end{Lemma}

\begin{proof}  Suppose that $M$ is a finitely generated $R$-module such that $\gamma(M)=\gamma(R)$. Then $\e(M)\geq 1$ by Corollary~\ref{HK Lower bound}. By Theorem~\ref{Forster-Swan}, for each $e\in\N$ there exists $P_e\in\Spec(R)$ such that $\mu_R(F^e_*M)\leq \mu_{R_{P_e}}(F^e_*M_{P_e})+\dim(R)$. By Proposition~\ref{Uniform Bound of HK} there is a constant $C\in\R$ such that for each $P\in\Spec(R)$, $\lambda_{R_P}(M_P/P^{[p^e]}M_P)\leq Cp^{e\height(P)}$. Equivalently, there is a constant $C\in\R$ such that for each $P\in\Spec(R)$, $\mu_{R_P}(F^e_*M_P)\leq Cp^{e(\height(P)+\alpha(P))}$. Suppose there existed an infinite subset $\Gamma\subseteq\N$ such that for each $e\in\Gamma$ the prime $P_e$ could be chosen such that $P_e\not\in Z_R$. Then for each $e\in\Gamma$, $$\mu_R(F^e_*M)\leq \mu_{R_{P_e}}(F^e_*M_{P_e})+\dim(R)\leq Cp^{e(\height(P)+\alpha(P))}+\dim(R)\leq Cp^{e(\gamma(R)-1)}+\dim(R).$$ Dividing by $p^{e\gamma(R)}$ and letting $e\in\Gamma\rightarrow \infty$ shows $\e(M)=0$.
\end{proof}

\begin{Theorem}\label{What is HK} Let $R$ be an F-finite ring and let $M$ be a finitely generated $R$-module. Then the following limits exist.
\begin{enumerate}
\item $\displaystyle \e(M)= \lim_{e\rightarrow \infty}\mu_R(F^e_*M)/p^{e\gamma(R)}$,
\item $\displaystyle \lim_{e\rightarrow \infty}\lambda_{R_{Q_e}}(M_{Q_e}/Q_e^{[p^e]}M_{Q_e})/p^{e\height(Q_e)}$, where $Q_e\in \Spec(R)$ is chosen such that $$\mu_{R_{Q_e}}(F^e_*M_{Q_e})=\max\{\mu_{R_P}(F^e_*M_P)\mid P\in\Spec(R)\},$$
\item $\displaystyle \lim_{e\rightarrow \infty}\e(M_{Q_e})$, where $Q_e\in \Spec(R)$ is chosen such that $$\mu(F^e_*M_{Q_e})=\max\{\mu_{R_P}(F^e_*M_P)\mid P\in\Spec(R)\}.$$
\end{enumerate}
The above limits agree, with the common value being $\displaystyle \max\{\e(M_P)\mid P\in Z_R\}$. 
 \end{Theorem}

\begin{proof}  It is clear that $\mu_R(F^e_*M)\geq \mu_{R_P}(F^e_*M_P)$ for all $P\in \Spec(R)$. So for every $P\in Z_R$, $$\frac{\mu(F^e_*M)}{p^{e\gamma(R)}}\geq \frac{\mu_{R_P}(F^e_*M_P)}{p^{e\gamma(R)}}=\frac{\lambda(M_P/P^{[p^e]}M_P)}{p^{e(\gamma(R)-\alpha(P))}}= \frac{\lambda(M_P/P^{[p^e]}M_P)} {p^{e\height(P)}}.$$  Letting $e\rightarrow \infty$ we see that $\e(M)\geq \e(M_P)$ for every $P\in Z_R$. This shows that $\e(M)\geq \max\{\e(M_P)\mid P\in Z_R\}$.

By Theorem~\ref{Uniform Convergence of HK} and Theorem~\ref{HK exists}, if $\epsilon>0$ then for $e\gg 0$, 
 \begin{enumerate}
\item  $\displaystyle \left|\frac{\lambda_{R_P}(M_P/P^{[p^e]}M_P)}{p^{e\height(P)}}-\e(M_P)\right|<\epsilon/3$ for all $P\in\Spec(R)$,
\item $\displaystyle \left|\frac{\mu_R(F^e_* M)}{p^{e\gamma(R)}}-\e(M)\right|<\epsilon/3$, and
\item $\displaystyle\frac{\dim(R)}{p^{e\gamma(R)}}<\epsilon/3$.
\end{enumerate}

For each $e>0$ let $Q_e\in\Spec(R)$ be such that $\max\{\mu_{R_P}(F^e_*M_P)\mid P\in\Spec(R)\}=\mu_{R_{Q_e}}(F^e_*M_{Q_e})$.   By Theorem~\ref{Forster-Swan} $\mu_R(F^e_* M)\leq \mu_{R_{Q_e}}(F^e_*M_{Q_e})+\dim(R)$ and by Lemma~\ref{Lemma for What is HK}, the prime $Q_e\in Z_R$ for all $e\gg 0$.

Therefore if  $e$ is suitably large,
\begin{align*}
\e(M_{Q_e})\leq\e(M)&<\frac{\mu_R(F^e_*M)}{p^{e\gamma(R)}}+\epsilon/3\leq \frac{\mu_{R_{Q_e}}(F^e_*M_{Q_e})}{p^{e\gamma(R)}}+\frac{\dim(R)}{p^{e\gamma(R)}}+\epsilon/3\\
&<\frac{\mu_{R_{Q_e}}(F^e_*M_{Q_e})}{p^{e\gamma(R)}}+\epsilon/3+\epsilon/3=\frac{\lambda_{R_{Q_e}}(M_{Q_e}/Q_e^{[p^e]}M_{Q_e})}{p^{e\height(P)}}+\epsilon/3+\epsilon/3\\
&<\e(M_{Q_e})+\epsilon/3+\epsilon/3+\epsilon/3=\e(M_{Q_e})+\epsilon.
\end{align*}
Thus $\e(M)\leq \max\{\e(M_P)\mid P\in Z_R\}$ and we must have equality. Furthermore, the above chain of inequalities shows that the limits in (2) and (3) exist and both converge to $\e(M)$. \end{proof}
 
\begin{Corollary}\label{associativity} Let $R$ be an F-finite ring such that $Z_R=V(Q)$ for some prime ideal $Q$,  e.g., $R$ is a domain. Then global Hilbert-Kunz multiplicity is additive on short exact sequences. Furthermore, if $M$ is any finitely generated $R$-module, then $\e(M)=\lambda_{R_Q}(M_Q)\e(R/Q)$.
\end{Corollary}

\begin{proof} Let  $\ell=\lambda_{R_Q}(M_Q)$. It is enough to show that  $\e(M)=\ell\e(R/Q)$. Corollary~\ref{Associativity Formula} shows that $\e(M)=\e(\bigoplus^\ell R/Q)$.  We can now use Theorem~\ref{What is HK} to conclude that $$\e(M)=\max\{\e(M_P)\mid P\in Z_R\}=\ell \max\{\e(R_P/QR_P)\mid P\in Z_R\}=\ell \e(R/Q).$$ The proof is complete. \end{proof}

\begin{Example}\label{Global HK and Z_R example} If $Z_R\neq \Spec(R)$, then global Hilbert-Kunz multiplicity is not an upper bound of $\{\e(R_P)\mid P\in\Spec(R)\}$. Let $K$ be an F-finite field and $(T,\m)$ a local F-finite domain such that $\gamma(K)>\gamma(T)$ and $\e(T)>1$ and let $R=K\times T$. Then $Z_R$ consists of the single prime $0\times T$, hence by Theorem~\ref{What is HK} $\e(R)=1< \e(R_{K\times \m})=\e(T)$.
\end{Example}

We now provide global analogues of Theorem~\ref{Watanabe Yoshida Hilbert-Kunz local}, Theorem~\ref{HK small means good local}, and Theorem~\ref{HK really small means regular local}. We remark that F-finite domains satisfy the hypotheses Lemma~\ref{Excellent Rings Lemma}, Theorem~\ref{Watanabe-Yoshida Hilbert-Kunz}, Theorem~\ref{HK small means good}, and Theorem~\ref{HK really small means regular}.

\begin{Lemma}\label{Excellent Rings Lemma} Let $R$ be an F-finite ring such that $Z_R=\Spec(R)$ and such that every associated prime of $R$ is minimal. Then for each $P\in\Spec(R)$, $R_P$ is formally unmixed. 
\end{Lemma}

\begin{proof} The assumption that $Z_R=\Spec(R)$ implies $Z_{R_P}=\Spec(R_P)$ for each $P\in\Spec(R)$. Hence $R$ is locally equidimensional by Lemma~\ref{Assh Lemma}. By Ratliff, the completion of an excellent equidimensional local ring is equidimensional, see \cite[Corollary B.4.3 and Theorem B.5.1]{HunekeSwanson}. As $R$ is excellent, $R_P\rightarrow \widehat{R_P}$ has regular fibers by \cite[Section~33, Lemma~4]{Matsumura1980}. In particular, all associated primes of $\widehat{R_P}$ are minimal, completing the proof of the Lemma.
\end{proof}

\begin{Theorem}\label{Watanabe-Yoshida Hilbert-Kunz} Let $R$ be an F-finite ring such that $Z_R=\Spec(R)$ and such that every associated prime of $R$ is minimal. Then $R$ is regular if and only if $\e(R)=1$.
\end{Theorem}

\begin{proof} By Theorem~\ref{What is HK}, $\e(R)=\max\{\e(R_P)\mid P\in \Spec(R)\}$. Hence $\e(R)=1$ if and only $\e(R_P)=1$ for each $P\in\Spec(R)$ if and only if $R_P$ is a regular local ring for each $P\in\Spec(R)$ by Lemma~\ref{Excellent Rings Lemma} and Theorem~\ref{Watanabe Yoshida Hilbert-Kunz local} if and only if $R$ is regular. \end{proof}

\begin{Theorem}\label{HK small means good} Suppose that $R$ is an F-finite ring such that $Z_R=\Spec(R)$ and such that every associated prime of $R$ is minimal. Let $e=\max\{e(R_P)\mid P\in \Spec(R)\}$, where $e(R_P)$ is the Hilbert-Samuel multiplicity of the local ring $R_P$. If $\e(R)\leq 1+\max\left\{1/\dim(R)!, 1/e\right\}$, then $R$ is strongly F-regular and Gorenstein.
\end{Theorem}

\begin{proof} By Theorem~\ref{What is HK}, $\e(R)=\max\{\e(R_P)\mid P\in\Spec(R)\}$.   One can now apply Lemma~\ref{Excellent Rings Lemma} and Theorem~\ref{HK small means good local} to know that for each $P\in\Spec(R)$ that $R_P$ is strongly F-regular and Gorenstein, hence $R$ is strongly F-regular and Gorenstein.  \end{proof}

\begin{Theorem}\label{HK really small means regular} Fix $d\in\N$. There is a number $\delta>0$ such that if $R$ is of dimension $d$, of any prime characteristic, F-finite, such that $Z_R=\Spec(R)$, such that all associated primes of $R$ are minimal, and $\e(R)\leq 1+\delta$, then $R$ is regular.
\end{Theorem}

\begin{proof} The proof is parallel to that of Theorem~\ref{HK small means good}. One only needs to reference Theorem~\ref{HK really small means regular local} instead of Theorem~\ref{HK small means good local}. Let $\delta(i)$ be a number as in Theorem~\ref{HK really small means regular local}, that works for rings of dimension $i$, and let $\delta = \min\{\delta(i) \mid i \leq d\}$.\end{proof}

\section{Global F-signature}\label{Global F-signature section}
 Let $R$ be an F-finite ring and $M$ a finitely generated $R$-module. Consider the following sequences of numbers.

\begin{enumerate}

\item Let $a_e(M)=\frk(F^e_*M)$ be the largest rank of a free summand appearing in the various direct sum decompositions of $F^e_*M$.  Then $a_e(M) \leq \rank(F^e_*M)\leq \rank(M)p^{e\gamma(R)}=O(p^{e\gamma(R)})$. 

\item Let $\tilde{a}_e(M)$ be the largest $\minrank$ of a projective summand appearing in various direct sum decompositions of $F^e_*M$. Then $a_e(M)\leq \tilde{a}_e(M)\leq \rank(F^e_*M)\leq \rank(M)p^{e\gamma(R)}=O(p^{e\gamma(R)})$. 
\end{enumerate}

\begin{Remark}\label{Remark due to Serre's Theorem} If $R$ is local, then $a_e(M)=\tilde{a}_e(M)$ and $\s(M)=\lim_{e\rightarrow \infty}a_e(M)/p^{e\gamma(R)}$. If $R$ is non-local, then Serre's Splitting Theorem, Theorem~\ref{Serre's Splitting Theorem}, shows that for each $e>0$ we have $a_e(M)\leq \tilde{a}_e(M)\leq a_e(M)+d$. Hence the limit $\lim_{e\rightarrow \infty}a_e(M)/p^{e\gamma(R)}$ exists if and only if the limit $\lim_{e\rightarrow \infty }\tilde{a}_e(M)/p^{e\gamma(R)}$ exists. Moreover, if the limits do exist then their limits are equal.
\end{Remark}

If $R$ is F-finite, not necessarily local, and $M$ is a finitely generated $R$-module, we define
\[
\ds \s(M) = \lim_{e \to \infty} \frac{a_e(M)}{p^{e \gamma(R)}}.
\]
We show in Theorem~\ref{Global F-signature exists} that the limit $\s(M)$ exists, and we call it the {\it global F-signature of $M$}. Note that, when $R$ is local, this is the usual definition of F-signature of a module $M$. 

\begin{Remark}\label{Reduce to the Reduced Case} Suppose that $R$ is an F-finite ring. The existence of a finitely generated $R$-module $M$ and $e>0$ such that $a_e(M)>0$ implies $a_e(R)>0$. In particular, $R$ is reduced. Recall the notation $Z_R = \{P \in \Spec(R) \mid \alpha(P) + \height(P) = \gamma(R)\}$ from Section~\ref{Global HK}. Observe that, if $Z_R \ne \Spec(R)$, then for any finitely generated $R$-module $M$ and any $P \notin Z_R$ we have $a_e(M)\leq a_e(M_P)\leq O(p^{e(\gamma(R)-1)})$. It follows that, in this case, $\s(M)=0$ for any finitely generated $R$-module $M$. These observations allow us to reduce our considerations to the scenario that $R$ is reduced and $Z_R = \Spec(R)$. In particular, $\Assh(R)=\Min(R)$. 

Suppose that $R$ is an F-finite reduced ring, and let $M$ be a finitely generated $R$-module. For each $P\in\Spec(R)$, we see $a_e(M)\leq a_e(M_P)$. Now further assume that $Z_R = \Spec(R)$. Then if there exists $P\in\Spec(R)$ such that $\s(M_P)=0$ then $\s(M)$ exists and is equal to $0$. Otherwise, $R$ will be strongly F-regular and hence a direct product of integral domains by Theorem~\ref{s>0 local}. We therefore reduce many considerations in this section to the case that $R$ is a domain.
\end{Remark}

\begin{Lemma}\label{Free Rank Lemma} Let $R$ be a Noetherian ring of finite Krull dimension. Let $M'\rightarrow M\rightarrow M''\rightarrow 0$ be a right exact sequence of finitely generated $R$-modules. Then $\frk_R(M)\leq \frk_R(M')+\mu_R(M'')+\dim(R)$.
\end{Lemma}
\begin{proof} For each $P\in\Spec(R)$, $\frk_{R_P}(M_P)\leq \frk_{R_P}(M'_P)+\mu_{R_P}(M''_P)$, see (\ref{Local Observation 2 part 2}) of Lemma~\ref{Local Observation 2}. In particular, $\frk_R(M)\leq \frk_{R_P}(M'_P)+\mu_R(M'')$ for each $P\in\Spec(R)$. By Theorem~\ref{Generalized Serre's Theorem} there is a prime $P\in\Spec(R)$ such that $\frk_{R_P}(M_P)\leq \frk_R(M)+\dim(R)$. Therefore $\frk_{R}(M)\leq \frk_R(M')+\mu_{R}(M'')+\dim(R)$.
\end{proof}

\begin{Lemma}\label{F-sig Lemma} Let $R$ be an F-finite ring and let $M,N$ be two finitely generated $R$-modules isomorphic at each prime $P\in\Assh(R)$. Then $a_e(M)=a_e(N)+O(p^{e(\gamma(R)-1)}).$
\end{Lemma}

\begin{proof} There are two right exact sequences $M\rightarrow N\rightarrow T_1\rightarrow 0$ and $N\rightarrow M\rightarrow T_2\rightarrow 0$ such that $(T_1)_P=(T_2)_P=0$ for each $P\in \Assh(R)$. By Lemma~\ref{Free Rank Lemma} $$|a_e(M)-a_e(N)|\leq \max\{\mu(F^e_*T_1)+\dim(R), \mu(F^e_*T_2)+\dim(R)\}.$$ The result follows from Lemma~\ref{Upper bound on generators}.
\end{proof}

\begin{Corollary}\label{Corollary to F-sig Lemma} Let $R$ be an F-finite ring and $0\rightarrow M'\rightarrow M\rightarrow M''\rightarrow 0$ a short exact sequence of finitely generated $R$-modules. Then $a_e(M)=a_e(M'\oplus M'')+O(p^{e(\gamma(R)-1)}).$
\end{Corollary}

\begin{proof} Without loss of generality, one may assume that $R$ is reduced. In particular, $M$ is isomorphic to $M'\oplus M''$ at all $P\in \Assh(R)$, and the result follows by Lemma~\ref{F-sig Lemma}.
\end{proof}

\begin{Example} As with global Hilbert-Kunz, one cannot expect $\s(M_1\oplus M_2)=\s(M_1)+\s(M_2)$. Let $R=\F_p\times \F_p$, $M_1=\F_p\times 0$, and $M_2=0\times \F_p$. Observe that $\gamma(R)=0$. Hence $\s(M_1)=\s(M_2)=0$ whereas $\s(R)=1$.
\end{Example}

\begin{Theorem}\label{Global F-signature exists} Let $R$ be an F-finite ring and $M$ a finitely generated $R$-module. Then the limit $\s(M)=\lim_{e\rightarrow \infty}a_e(M)/p^{e\gamma(R)}$ exists. Moreover, there exists a constant $C\in\R$ such that for each $e\in\N$, $a_e(M)\leq \s(M)p^{e\gamma(R)}+Cp^{e(\gamma(R)-1)}$.
\end{Theorem}

\begin{proof} Without loss of generality, one may assume that $R$ is reduced and $\alpha(P)+\height(P)=\gamma(R)$ for each $P\in\Spec(R)$. By considering a prime filtration of $M$, repeated use of Corollary~\ref{Corollary to F-sig Lemma} allows one to reduce all considerations to the scenario that $M\cong R/Q_1\oplus \cdots \oplus R/Q_\ell$ where $Q_i\in\Min(R)$. As $R$ is a reduced and $\alpha(P)+\height(P)=\gamma(R)$ for each $P\in\Spec(R)$, there is a short exact sequence $$0\rightarrow  F_*M\rightarrow M^{\oplus p^\gamma(R)}\rightarrow  T\rightarrow 0$$ so that $T_Q=0$ for each $Q\in\Min(R)$. For each $e\in \N$ $$a_e(M^{\oplus p^{\gamma(R)}})\leq a_{e+1}(M)+\mu(F^e_*T)+\dim(R)$$ by Lemma~\ref{Free Rank Lemma}. By Corollary~\ref{Upper bound on generators} there is a constant $C\in\R$ such that $\mu(F^e_*T)\leq Cp^{e(\gamma(R)-1)}$. Note $p^{\gamma(R)}a_e(M)\leq a_e(M^{\oplus p^{\gamma(R)}})$, dividing the above inequality by $p^{(e+1)\gamma(R)}$ and applying crude estimates, $$\frac{a_e(M)}{p^{e\gamma(R)}}\leq \frac{a_{e+1}(M)}{p^{(e+1)\gamma(R)}}+\frac{C+\dim(R)}{p^e}.$$ The theorem follows from (\ref{Sequence Lemma 2}) of Lemma~\ref{Sequence Lemma}. \end{proof}

\begin{Lemma}\label{Application of Stafford's Theorem} Let $R$ be a Noetherian ring of finite Krull dimension, of any characteristic. Suppose that $M$ is a finitely generated $R$-module. Then for each $n\in\N$, $$|\frk(M^{\oplus n})-n\frk(M)|\leq n\dim(R).$$
\end{Lemma}
\begin{proof} It is clear that $n\frk(M)\leq \frk(M^{\oplus n})$. By Theorem~\ref{Generalized Serre's Theorem} there exists a $P\in\Spec(R)$ such that $\frk(M_P)\leq \frk(M)+\dim(R)$. Hence $\frk(M^{\oplus n})\leq \frk_{R_P}(M^{\oplus n}_{P})=n\frk(M_P)\leq n\frk(M)+n\dim(R).$
\end{proof}

\begin{Theorem}\label{F-signature Growth Rate} Let $R$ be an F-finite ring and $M$ a finitely generated $R$-module. Then $a_e(M)=\s(M)p^{e\gamma(R)}+O(p^{e(\gamma(R)-1)}).$
\end{Theorem}

\begin{proof} Without loss of generality, we may assume that $R$ is reduced and $\alpha(P)+\height(P)=\gamma(R)$ for each $P\in\Spec(R)$. As in the proof of Theorem~\ref{Global F-signature exists}, we may assume $M\cong R/Q_1\oplus \cdots \oplus R/Q_\ell$ where $Q_i\in\Min(R)$. In this case, there is a short exact sequence $$0\rightarrow M^{\oplus p^\gamma(R)}\rightarrow F_*M\rightarrow T\rightarrow 0$$ so that $T_P=0$ for each $P\in\Min(R)$. For each $e\in\N$ $$a_{e+1}(M)\leq a_e(M^{\oplus p^{\gamma(R)}})+\mu(F^e_*T)+\dim(R)$$ by Lemma~\ref{Free Rank Lemma}. By Corollary~\ref{Upper bound on generators} there is a constant $C\in\R$ such that for each $e\in\N$ $\mu(F^e_*T)\leq Cp^{e(\gamma(R)-1)}$. Hence by Lemma~\ref{Application of Stafford's Theorem}, $$a_{e+1}(M)\leq p^{\gamma(R)}a_e(M)+Cp^{e(\gamma(R)-1)}+p^{\gamma(R)}\dim(R)+\dim(R).$$ Dividing by $p^{(e+1)\gamma(R)}$ and applying a crude estimate shows $$\frac{a_{e+1}(M)}{p^{(e+1)\gamma(R)}}\leq \frac{a_e(M)}{p^{e\gamma(R)}}+\frac{C+2\dim(R)}{p^e}.$$ The theorem follows from Theorem~\ref{Global F-signature exists} and (\ref{Sequence Lemma 3}) of Lemma~\ref{Sequence Lemma}.
\end{proof}

\begin{Lemma}\label{Splitting Lemma 1} Let $R$ be an F-finite ring of dimension $d$ and $M$ a finitely generated $R$-module.  For each $e\in\N$ choose a decomposition $F^e_*M\cong R^{\oplus n_e}\oplus M_e$ such that $M_e$ does not have a free summand. There exists $Q\in\Spec(R)$ such that $\frk(F^e_*M_Q)\leq n_e+\dim(R)$.
\end{Lemma}

\begin{proof} By Theorem~\ref{Generalized Serre's Theorem} there is a $Q\in\Spec(R)$ such that $\frk((M_e)_Q)\leq d$. The desired result now follows since $\frk(F^e_*M_Q)=n_e+\frk((M_e)_Q)$. \end{proof}

\begin{Lemma}\label{Splitting Lemma 2} Let $R$ be an F-finite ring of dimension $d$ and let $M$ be a finitely generated $R$-module. For each $e\in\N$ choose a decomposition $F^e_*M\cong \Omega_e\oplus M_e$ such that $\Omega_e$ is projective of min-rank $m_e$ and $M_e$ does not have a projective summand. Then there exists $Q\in\Spec(R)$ such that $\frk(F^e_*M_Q)\leq m_e+\dim(R)$.
\end{Lemma}

\begin{proof} By Theorem~\ref{Generalized Serre's Theorem} there is a $Q\in\Spec(R)$ such that $\frk((M_e)_Q)\leq d$, else $M_e$ has a free, and hence projective, summand. The desired result now follows as in Lemma~\ref{Splitting Lemma 1}. \end{proof}

\begin{Lemma}\label{Convergence Lemma} Let $R$ be an F-finite ring such that $Z_R = \Spec(R)$ and $M$ a finitely generated $R$-module. For each $e>0$, let $Q_e\in\Spec(R)$ be such that $a_e(M_{Q_e})=\min\{a_e(M_P)\mid P\in\Spec(R)\}$. Then both  $\frac{a_e(M_{Q_e})}{p^{e\gamma(R)}}$ and $\s(M_{Q_e})$ converge to $\min\{\s(M_P)\mid P\in\Spec(R)\}$ as $e\rightarrow \infty$.
\end{Lemma}

\begin{proof} Since $Z_R = \Spec(R)$, the F-signature function $\Spec(R)\rightarrow \R$ sending $P\mapsto \s(R_P)$ is lower semi-continuous by Theorem~\ref{Lower semi-continuity of the F-signature}. 
Therefore there exists $Q\in\Spec(R)$ such that $\s(M_Q)=\inf\{\s(M_P)\mid P\in\Spec(R)\}$. By Theorem~\ref{Uniform Convergence of a_e} and Remark~\ref{Uniform Convergence Remark} the functions $\s_e:\Spec(R)\rightarrow \R$ sending a prime $P\mapsto \s_e(M_P)= a_e(M_P)/p^{e\gamma(R)}$ converge uniformly to their limit, namely $\s:\Spec(R)\rightarrow \R$ sending a prime $P\mapsto \s(M_P)$, the F-signature of $M_P$. Let $\epsilon>0$ and $e_0\gg 0$ such that for $e\geq e_0$, $\left|\s_e(M_P)-\s(M_P)\right|<\epsilon/2$ for every $P\in\Spec(R)$. Then for $e\geq e_0$
\begin{align*}
\s(M_Q)\leq \s(M_{Q_e})< \s_e(M_{Q_e})+\epsilon/2\leq \s_e(M_Q)+\epsilon/2<\s(M_Q)+\epsilon/2+\epsilon/2=\s(M_Q)+\epsilon.
\end{align*}
The lemma now follows. \end{proof}

\begin{Theorem}\label{What is Global F-signature} Let $R$ be an F-finite ring such that $Z_R = \Spec(R)$, and $M$ a finitely generated $R$-module. Then the following limits exist:

\begin{enumerate}
\item $\displaystyle \s(M)= \lim_{e\rightarrow \infty} \frac{a_e(M)}{p^{e\gamma(R)}}$,
\item $\displaystyle \s(M)= \lim_{e\rightarrow \infty} \frac{\tilde{a}_e(M)}{p^{e\gamma(R)}}$,
\item $\displaystyle \lim_{e\rightarrow \infty}\frac{n_e}{p^{e\gamma(R)}}$, where $n_e$ is the rank of a free direct summand of $F^e_*M$ appearing in a choice of decomposition $F^e_*M\cong R^{\oplus n_e}\oplus M_e$, where $M_e$ has no free summand,
\item $\displaystyle \lim_{e\rightarrow \infty}\frac{m_e}{\rank(F^e_*R)}$, where $m_e$ is the min-rank of a project summand $\Omega_e$ of $F^e_*M$ appearing in a choice of decomposition $F^e_*M\cong \Omega_e\oplus M_e$, where $M_e$ has no projective summand,
\item $\displaystyle \lim_{e\rightarrow \infty }\frac{a_e(M_{Q_e})}{p^{e\gamma(R)}}$, where $Q_e\in\Spec (R)$ is chosen such that $$a_e(M_{Q_e})=\min\{a_e(M_P)\mid P\in\Spec(R)\},$$
\item $\displaystyle \lim_{e\rightarrow \infty} \s(M_{Q_e})$, where $Q_e\in\Spec (R)$ is chosen such that $$a_e(M_{Q_e})=\min\{a_e(M_P)\mid P\in\Spec(R)\}.$$
\end{enumerate}
Moreover, all of the above limits agree, with common value being $\displaystyle\min\{\s(M_P)\mid P\in\Spec(R)\}$. 
\end{Theorem}

\begin{proof}  The existence and agreements of the limits in (1) and (2) is the content of Theorem~\ref{Global F-signature exists} and Remark~\ref{Remark due to Serre's Theorem}. The existence and agreements of the limits in (5) and (6) is the content of Lemma~\ref{Convergence Lemma}. The proof of the theorem is easily reduced to showing the convergence of the sequences in (3) and (4) to $\min\{\s(M_P)\mid P\in\Spec(R)\}$.  Let $Q_e\in\Spec(R)$ be as in Lemma~\ref{Convergence Lemma}. Then by Lemmas~\ref{Splitting Lemma 1} and~\ref{Splitting Lemma 2}, $a_e(M_{Q_e})\leq n_e+d$ and $a_e(M_{Q_e})\leq m_e+d$. Observe that $m_e,n_e\leq a_e(M_{Q_e})$. Therefore $\frac{a_e(M_{Q_e})-d}{p^{e\gamma(R)}}\leq \frac{n_e}{p^{e\gamma(R)}}\leq \frac{a_e(M_{Q_e})}{p^{e\gamma(R)}}$ and $\frac{a_e(M_{Q_e})-d}{p^{e\gamma(R)}}\leq \frac{m_e}{p^{e\gamma(R)}}\leq \frac{a_e(M_{Q_e})}{p^{e\gamma(R)}}$. By Lemma~\ref{Convergence Lemma}, $\frac{n_e}{p^{e\gamma(R)}}$ and $\frac{m_e}{p^{e\gamma(R)}}$ must converge to $\min\{\s(M_P)\mid P\in\Spec(R)\}$. \end{proof}

\begin{Corollary}\label{What is F-sig Corollary} Let $R$ be an F-finite ring such that $Z_R = \Spec(R)$, and let $M$ be a finitely generated $R$-module. Then $\s(M) = \min\{\rank_{R_P}(M_P)\s(R_P) \mid P \in \Spec(R)\}$. In addition, if $R$ is a domain, then $\s(M)=\rank_R(M)\s(R)$.
\end{Corollary}

\begin{proof} By Theorem~\ref{What is Global F-signature}, $\s(M)=\min\{\s(M_P)\mid P\in \Spec(R)\}$. For each $P\in\Spec(R)$, $\s(M_P)=\rank_{R_P}(M_P)\s(R_P)$, and the first claim follows.  If $R$ is a domain, we have that $\rank_{R_P}(M_P) = \rank_R(M)$ for any $P \in \Spec(R)$. Thus, in this case, we have $\s(M)=\rank_R(M)\min\{\s(R_P)\mid P\in\Spec(R)\}$, which is $\rank(M)\s(R)$ by a repeated application of Theorem~\ref{What is Global F-signature}.
\end{proof}

\begin{Theorem}\label{s=1}  Let $R$ be an F-finite ring such that $Z_R = \Spec(R)$. Then $\s(R)=1$ if and only if $R$ is regular.
\end{Theorem}

\begin{proof} The ring $R$ is regular if and only if for each $Q\in \Spec(R)$ the local ring $R_Q$ is a regular local ring. The local ring $R_Q$ is regular if and only if $\s(R_Q)=1$ by Theorem~\ref{s=1 local}. By Theorem~\ref{What is Global F-signature} this will happen if and only if $\s(R)=1$. \end{proof}

\begin{Theorem}\label{s>0} Let $R$ be an F-finite ring such that $Z_R = \Spec(R)$. Then $\s(R)>0$ if and only if $R$ is strongly F-regular.
\end{Theorem}

\begin{proof} An F-finite ring is strongly F-regular if and only if each localization of $R$ at a prime ideal is strongly F-regular. This is equivalent to $\s(R_Q)>0$ for each $Q\in\Spec(R)$ by Theorem~\ref{s>0 local}. This is equivalent to $\s(R)=\min\{\s(R_P)\mid P\in\Spec(R)\}>0$.  \end{proof}

\begin{Example} If $Z_R \ne \Spec(R)$, i.e., if there exists $P\in\Spec(R)$ such that $\alpha(P)+\height(P)\not=\gamma(R)$, then $\s(R)=1$ is not equivalent to $R$ being regular and $\s(R)>0$ is not equivalent to $R$ being strongly F-regular. Let $R=\F_p\times \F_p(t)$. Then $R$ is regular, hence strongly F-regular. But $\alpha(P)+\height(P)$ varies at the two different prime ideals of $R$, hence $\s(R)=0$ by Remark~\ref{Reduce to the Reduced Case}.
\end{Example}

\begin{Theorem}\label{F-signature big means regular global} Fix $d\in \N$. There is a number $\delta>0$ such that, if $R$ is an F-finite ring of dimension $\dim(R) \leq d$, of any prime characteristic, and such that $\s(R)\geq 1-\delta$, then $R$ is regular.
\end{Theorem}

\begin{proof} Let $\delta(i)$ be a number as in Theorem~\ref{F-signature big means regular}, that works for rings of dimension $i$, and let $\delta = \min\{\delta(i) \mid i \leq d\}$. Without loss of generality, we may assume $\s(R)>0$, thus we may assume that $Z_R = \Spec(R)$. If $\s(R)\geq 1-\delta$, then $\s(R_P)\geq 1-\delta$ for each $P\in\Spec(R)$. It follows that $R_P$ is regular for each $P\in\Spec(R)$, that is, $R$ is regular.
\end{proof} 

\subsection{Global F-signature of a Cartier subalgebra} In what follows, $R$ is an F-finite ring and $\D$ is a Cartier subalgebra. Given a choice of direct summand $M$ of $F^e_*R$, with splitting $M\subseteq F^e_*R\rightarrow M$, we say that a summand $N$ of $M$ is a $\D$-summand if $N\cong R^{\oplus n}$ is free and the natural projection map $F^e_*R\rightarrow M\rightarrow N$ is a direct sum of elements of $\D_e$. The choice of isomorphism $N\cong R^{\oplus n }$ does not change whether or not $N$ is a $\D$-summand. We denote by $a(M,\D)$ the largest rank of a $\D$-summand appearing in various direct sum decompositions of $M$. Recall that $a(F^e_*R, \D)=a^\D_e(R)$ is the usual $e$th Frobenius splitting number of the pair $(R,\D)$, see Section~\ref{Background}.

\begin{Lemma}\label{Cartier Lemma 1} Let $R$ be an F-finite ring and $M$ be a direct summand of $F^e_*R$. Suppose that $x\in M$ and that $(Rx)_Q\subseteq M_Q$ is a $\D_Q$-summand for each $Q\in\Spec(R)$. Then $Rx\subseteq M$ is a $\D$-summand.
\end{Lemma}

\begin{proof} Our assumptions allow us to find $s_1,\dots,s_n\in R$ such that $(s_1,\dots,s_n)=R$ and such that $(Rx)_{s_i}\subseteq M_{s_i}$ is a $\D_{s_i}$-summand. After replacing $s_i$ by powers of themselves, we can find $\varphi_1,\dots,\varphi_n\in \D_e$ such that $\varphi_i(x)=s_i$. There are elements $r_1,\dots,r_n\in R$ such that $r_1s_1+\cdots +r_ns_n=1$. Let $\varphi= r_1\varphi_1+\cdots +r_n\varphi _n\in \D_e$, then $\varphi(x)=1$. \end{proof}

\begin{Condition}\label{Condition dagger} We will say that $(R,\D)$ satisfies condition $(\dagger)$ if at least one of the following conditions is satisfied:
\begin{itemize}
\item Condition $(*)$ from Condition~\ref{Condtion *},
\item $\D=\C^{\a^t}$ for some ideal $\a\subseteq R$ and some $t>0$, 
\item $R$ is normal and $\D=\C^{(R,\Delta)}$ for some effective $\Q$-divisor $\Delta$.
\end{itemize}
\end{Condition}

\begin{Lemma}\label{Convergence Lemma for Cartier algebras} Let $R$ be an F-finite domain and $\D$ a Cartier subalgebra. For each $e\in \Gamma_\D$, let $Q_e\in\Spec(R)$ be such that $a_e(R_{Q_e},\D_{Q_e})=\min\{a_e(R_P,\D_P)\mid P\in\Spec(R)\}$. Then the sequence  $\s_e(R_{Q_e},\D_{Q_e})$ converges to a limit as $e\in\Gamma_\D\rightarrow \infty$. Moreover, if $(R,\D)$ satisfies condition $(\dagger)$, then the limit converges to $\min\{\s(R_P,\D_P)\mid P\in\Spec(R)\}$.
\end{Lemma}

\begin{proof} By Theorem~\ref{Uniform Convergence of a_e Cartier}, there is a constant $C\in\R$ such that for each $e,e'\in\Gamma_\D$ and each $P\in\Spec(R)$, $$\s_e(R_P,\D_P)-\s_{e+e'}(R_P,\D_P)<\frac{C}{p^e}.$$ It follows that, for each $e,e'\in\Gamma_\D$, we have
$$ \s_e(R_{Q_e},\D_{Q_e})\leq \s_e(R_{Q_{e+e'}},\D_{Q_{e+e'}})\leq \s_{e+e'}(R_{Q_{e+e'}},\D_{Q_{e+e'}})+\frac{C}{p^e},$$ and we conclude that the limit $\lim_{e\rightarrow \infty}\s_e(R_{Q_e},\D_{Q_e})$ exists by (\ref{Sequence Lemma 2}) of Lemma~\ref{Sequence Lemma}.

Now assume that $(R,\D)$ satisfies $(\dagger)$. Then, Theorem~\ref{Uniform Convergence of a_e Cartier} and Theorem~\ref{PTUniformConvergence} imply that the functions $\s_e:\Spec(R)\rightarrow \R$, defined as $Q\mapsto \s_e(R_Q,\D_Q)$, converge uniformly to their limit, namely $\s:\Spec(R)\rightarrow \R$ sending a prime $Q$ to the the F-signature $\s(R_Q,\D_Q)$ of the pair $(R_Q,\D_Q)$. This allows one to proceed as in the proof of  Lemma~\ref{Convergence Lemma}.
\end{proof}

We say that a projective summand $\Omega$ of $F^e_*R$ is a $\D$-summand if $a(\Omega_Q,\D_Q)=\rank(\Omega_Q)$ for each $Q\in\Spec(R)$. We call a projective summand $\Omega$ of $F^e_*R$ a free $\D$-summand if $\Omega$ is free and a $\D$-summand. Let $a_e(R,\D)$ be the largest rank of a free $\D$-summand appearing in various direct sum decompositions of $F^e_*R$, and denote by $\tilde{a}_e(R,\D)$ the largest min-rank of a projective $\D$-summand appearing in various direct sum decompositions of $F^e_*R$. We define the {\it global F-signature of the pair $(R,\D)$} as
\[
\ds \s(R,\D) = \lim_{e \in \Gamma_\D \to \infty} \frac{a_e(R,\D)}{p^{e\gamma(R)}}.
\]
We show the existence of this limit in the following theorem, and we relate it with other limits as in Theorem~\ref{What is Global F-signature}.

\begin{Theorem}\label{Global F-signature for Cartier algebras} Let $R$ be an F-finite domain of dimension $d$ and let $\D$ be a Cartier subalgebra. Then the following limits exist:

\begin{enumerate}
\item $\s(R,\D) = \displaystyle \lim_{e\in\Gamma_\D\rightarrow \infty} \frac{a_e(R,\D)}{p^{e\gamma(R)}}$,
\item $\displaystyle \lim_{e\in\Gamma_\D\rightarrow \infty} \frac{\tilde{a}_e(R,\D)}{p^{e\gamma(R)}}$,
\item $\displaystyle \lim_{e\in\Gamma_\D\rightarrow \infty}\frac{n_e}{p^{e\gamma(R)}}$, where $n_e$ is the rank of a free  $\D$-summand of $F^e_*R$ appearing in a choice of decomposition $F^e_*R\cong R^{n_e}\oplus M_e$ where $M_e$ has no free $\D$-summand,
\item $\displaystyle \lim_{e\in\Gamma_\D\rightarrow \infty}\frac{m_e}{p^{e\gamma(R)}}$, where $m_e$ is the min-rank of a project $\D$-summand $\Omega_e$ of $F^e_*R$ appearing in a choice of decomposition $F^e_*R\cong \Omega_e\oplus M_e$ where $M_e$ has no projective $\D$-summand,
\item $\displaystyle \lim_{e\in\Gamma_\D\rightarrow \infty }\frac{a_e(R_{Q_e},\D_{Q_e})}{p^{e\gamma(R)}}$, where $Q_e\in\Spec (R)$ is chosen such that $$a_e(R_{Q_e},\D_{Q_e})=\min\{a_e(R_P,\D_P)\mid P\in\Spec(R)\}.$$
\end{enumerate}
Moreover, all of the above limits agree. 
If $(R,\D)$ satisfies condition $(\dagger)$, then all the above limits equal $\min\{\s(R_P,\D_P)\mid P\in \Spec(R)\}$.
\end{Theorem}

\begin{proof} The convergence of the limit in (5) is the content of Lemma~\ref{Convergence Lemma for Cartier algebras}. Suppose that $n_e$ and $m_e$ are as in $(3)$ and $(4)$. Then Theorem~\ref{Local to global splitting numbers Cartier} easily implies that $m_e\leq a_e(R_{Q_e},\D_{Q_e})\leq m_e+d$ and $n_e \leq a_e(R_{Q_e},\D_{Q_e})\leq n_e+d$. It follows that the limits in $(1)-(4)$ all exist and are equal to the limit in $(5)$. If we assume that $(R,\D)$ satisfies $(\dagger)$, then Lemma~\ref{Convergence Lemma for Cartier algebras} implies that the common limit value is indeed $\min\{\s(R_P,\D_P)\mid P\in \Spec(R)\}$. \end{proof}

\begin{Corollary}\label{s>0 Cartier}  Let $R$ be an F-finite domain and let $\D$ be a Cartier algebra satisfying condition $(\dagger)$. Then $\s(R,\D)>0$ if and only if $(R,\D)$ is strongly F-regular.
\end{Corollary}

\begin{proof} A pair $(R,\D)$ is strongly F-regular if and only if for each $P\in\Spec(R)$ the pair $(R_P,\D_P)$ is strongly F-regular. Positivity of $\s(R_P,\D_P)$ is equivalent to strong F-regularity of $(R_P,\D_P)$ by Theorem~\ref{s>0 Cartier local}. By Theorem~\ref{Lower semi-continuity of the F-signature Cartier} and Theorem~\ref{Global F-signature for Cartier algebras} there is a $Q\in\Spec(R)$ such that $\s(R,\D)=\s(R_Q,\D_Q)$.  \end{proof}

Corollary~\ref{s>0 Cartier} brings up the following natural question.

\begin{Question}\label{Positivity Question}Let $R$ be an F-finite domain and $\D$ a Cartier subalgebra. Is positivity of $\s(R,\D)$ equivalent to strong F-regularity of $\D$?
\end{Question}
Suppose that $R$ is an F-finite domain and $\D$ a Cartier subalgebra. Suppose that one could show that the functions $\s_e:\Spec(R)\rightarrow \R$ sending $P\mapsto \s_e(R_P,\D_P)$ converge uniformly to their limit function, namely $\s:\Spec(R)\rightarrow \R$ which sends $P\mapsto \s(R_P,\D_P)$. Then one can follow the methods of Theorem~\ref{What is Global F-signature} to establish $\s(R,\D)=\min\{\s(R_P,\D_P)\mid P\in\Spec(R)\}$. Such a result would establish a positive answer to Question~\ref{Positivity Question}. We therefore ask the following more specific question.

\begin{Question}Suppose that $R$ is an F-finite domain and $\D$ a Cartier subalgebra. Do the functions $\s_e:\Spec(R)\rightarrow \R$ sending $P\mapsto \s_e(R_P,\D_P)$ converge uniformly to their limit as $e\in\Gamma_\D\rightarrow \infty$?
\end{Question}

\section{Global F-invariants under faithfully flat extensions}\label{Global non-F-finite}

\subsection{Global F-signature} We now study the behavior of global F-signature under faithfully flat extensions. Recall that if $R$ is an F-finite ring, then we let $Z_R=\{P\in\Spec(R)\mid \alpha(P)+\height(P)=\gamma(R)\}$. Let $M$ be a finitely generated $R$-module. Remark~\ref{Reduce to the Reduced Case} and Theorem~\ref{What is Global F-signature} combined state that $\s(M)=0$ if $Z_R\not=\Spec(R)$, and that $\s(M)=\min\{\s(M_P)\mid P\in\Spec(R)\}$ if $Z_R=\Spec(R)$.

\begin{Theorem}\label{Global F-signature flat} Let $R\rightarrow T$ be a faithfully flat map of F-finite rings such that $Z_R=\Spec(R)$ and $Z_T=\Spec(T)$, and $M$ a finitely generated $R$-module. Then $\s(M)\geq \s(M \otimes_R T)$. If moreover the closed fibers of $R\rightarrow T$ are regular, then $\s(M)=\s(M \otimes_R T)$.
\end{Theorem}

\begin{proof} By Theorem~\ref{What is Global F-signature}, $\s(M)=\min\{\s(M_P)\mid P\in\Spec(R)\}$ and $\s(M\otimes_R T)=\min\{\s(M\otimes _R T_Q)\mid Q\in\Spec(T)\}$. Let $P\in\Spec(R)$ be such that $\s(M)=\s(M_P)$ and let $Q\in\Spec(T)$ be such that $Q\cap R =P$. By Theorem~\ref{F-signature comparison}, $\s(M_P)\geq \s(M\otimes_R T_Q)$, hence $\s(M)\geq \s(M\otimes_R T)$. 

Suppose that $R\rightarrow T$ has regular closed fibers and let $Q\in\Spec(T)$ be such that $\s(M\otimes_R T)=\s(M\otimes_R T_Q)$. If $\m$ is a maximal ideal of $T$ containing $Q$, then $\s(M\otimes_R T_\m)\leq \s(M\otimes_R T_Q)$. Thus without loss of generality we may assume that $Q$ is maximal, thus $P=R\cap Q$ is maximal in $R$. By Theorem~\ref{F-signature comparison}, $\s(M_P)=\s(M\otimes_R T_Q)$, it follows that $\s(M)=\s(M\otimes_R T)$.
\end{proof}

Suppose that $R\rightarrow T$ is a faithfully flat extension of F-finite rings satisfying the hypotheses of Theorem~\ref{Global F-signature flat}. Example~\ref{Faithfully flat example} below shows that it need not be the case that $a_e(R)/p^{e\gamma(R)}\geq a_e(T)/p^{e\gamma(T)}$, even though the inequality holds after taking limits. One should compare this to the local situation in Theorem~\ref{F-signature comparison}. Before providing such an example, we first discuss the existence of an F-finite regular ring $R$ such that $F^e_*R$ is not free. The class of examples we discuss were already known to exist by experts.\footnote{The class of examples we discuss in Example~\ref{Flat but not free} were communicated to us by Florian Enescu. Florian Enescu learned of such examples from Mohan Kumar.}

\begin{Example}\label{Flat but not free} If $R$ is a regular F-finite domain, then $F^e_*R$ need not be free as an $R$-module. Let $k$ be an algebraic closed field of characteristic $p$, $X$ an elliptic curve over $k$, as in \cite[Chapter 4.4]{Hartshorne}, $x_0\in X$ be a chosen point for the group law on $X$, and let $K(X)$ be the function field of $X$. Assume $X$ is ordinary, that is the Frobenius morphism $F:X\rightarrow X$ induces an injective map of $1$-dimensional vector spaces $H^1(X,\O_X)\rightarrow H^1(X,\O_X)$. The assumption that $X$ is ordinary guarantees that the map of structure sheaves $\O_X\rightarrow F^e_*\O_X$ is split. Denote by $\mathcal{E}$ the cokernel of $\O_X\rightarrow F^e_*\O_X$. Then $\mathcal{E}\cong \O_X(x_1-x_0)\oplus \cdots \oplus \O_X(x_{p^e-1}-x_0)$ where $x_0,x_1,\dots,x_{p^e-1}$ are the $p^e$ distinct $p^e$ torsion points of $X$, see \cite[Example 2.18, Exercise 2.19]{PST} for further details.

If $\char k\not=2$ or if $e>0$ let $U=X-\{x_1,\dots,x_{p^e-2}\}$. If $\char k=2$ and $e=1$ let $U=X-\{x_2\}$ for some point $x_2$ which is not a $2$-torsion point of $X$.  As $X$ is a non-singular projective curve, $U$ is an open affine set. Let $R=\Gamma(U, \O_X)$ and $M=\Gamma(U, \O_X(x_{p^e-1}-x_0))$, then $F^e_*R\cong R^{\oplus p^e-1}\oplus M$ is projective of rank $p^e$. By examining the $p^e$th exterior product of $R^{\oplus p^e-1}\oplus M$, one sees that $F^e_*R$ is a free $R$-module of rank $p^e$ if and only if $M$ is a free module of rank $1$. We claim that $M$ is not free. Else, $M$ is identified with $R\cdot f$ for some $f\in K(X)$. Equivalently, the divisor $x_{p^e-1}-x_0$ is linearly equivalent to $0$ on $U$. As $x_0,x_{p^e-1}\not\in U$, this will imply $x_{p^e-1}-x_0$ is linearly equivalent to $0$ on $X$, contradicting that $x_0, x_{p^e-1}$ are distinct points.
\end{Example}

\begin{Example}\label{Faithfully flat example} Suppose that $R\rightarrow T$ is a faithfully flat map of F-finite domains. Then it does not necessarily follow that $a_e(R)/p^{e\gamma(R)}\geq a_e(T)/p^{e\gamma(T)}$ for each $e\in\N$, even though the inequality holds after taking limits. Let $R$ be a Dedekind domain affine over the algebraically closed field $k$ of characteristic $p$. Then $F^e_*R$ is projective of rank $p^e$. By Theorem~\ref{Serre's Splitting Theorem}, $p^e-1\leq a_e(R)\leq p^e$ with $a_e(R)=p^e$ if and only if $F^e_*R$ is free. Let $R$ be as in Example~\ref{Flat but not free}, so that $F^e_*R$ is not free. Consider the faithfully flat extension $R\rightarrow R[t]\rightarrow T=R[t]_W$ where $W$ is the multiplicative set $R[t]-\cup_{\m\in\Max(R)}\m R[t]$. Observe that $T$ is a Dedekind domain and $F^e_*T$ is projective of rank $p^{2e}$. By Theorem~\ref{Serre's Splitting Theorem}, $a_e(T)$ is either $p^{2e}-1$ or $p^{2e}$. Then $a_e(R)/p^{e\gamma(R)}=\frac{p^e-1}{p^e}<\frac{p^{2e}-1}{p^{2e}}\leq a_e(T)/p^{e\gamma(T)}$.
\end{Example}

We now discuss the behavior of global F-signature of F-finite faithfully flat extensions of rings which are either F-finite or essentially of finite type over an excellent local ring. Recall that, by \cite{Yao2006}, given any $d$-dimensional local ring $(R,\m,k)$ of prime characteristic and a finitely generated $R$-module $M$, we can define a sequence $\# (F^e_*M)/p^{e d}$ that agrees with $a_e(M)/p^{e\gamma(R)}$ when $R$ is F-finite. We still denote an element of this sequence by $\s_e(M)$, even when $R$ is not F-finite. Let $R$ be either F-finite or essentially of finite type over an excellent local ring and let $M$ a finitely generated $R$-module. We define the {\it local-minimal F-signature of $M$} as 
\[
\ds \min\{\s(M_P) \mid P \in \Spec(R)\} = \min\{\rank_{R_P}(M_P)\s(R_P) \mid P \in \Spec(R)\},
\]
and we denote it by $\ts(M)$. We note that such a minimum exists, since in our assumptions the F-signature function $\s: \Spec(R) \to \R$, sending $P \mapsto \s(R_P)$, is lower semi-continuous by Corollary~\ref{Lower semi-continuity of the F-signature}. In particular, $R$ is strongly F-regular if and only if $\ts(R)>0$ by Theorem~\ref{s>0 local}. Observe that, when $R$ is F-finite and $Z_R=\Spec(R)$, $\ts(M)$ coincides with the global F-signature $\s(M)$ defined in Section~\ref{Global F-signature section}. See Theorem~\ref{What is Global F-signature}. 

In Theorem~\ref{Global F-signature non-F-finite}, we show equality between $\ts(M)$, $$\sup\{\ts(T\otimes_R M)\mid R\rightarrow T \mbox{ is faithfully flat and $T$ is F-finite}\},$$ and $$\sup\{\s(T\otimes_R M)\mid R\rightarrow T \mbox{ is faithfully flat and $T$ is F-finite}\}.$$ We begin  with a lemma.

\begin{Lemma}
\label{Locally Equidimensional Lemma} Let $R$ be an F-finite locally equidimensional ring. Then there is a faithfully flat extension $R\rightarrow T$ with regular fibers such that $T$ is F-finite, $\gamma(T)=\gamma(R)$, and $Z_T=\Spec(T)$.
\end{Lemma}

\begin{proof}
By Lemma~\ref{Kunz's Lemma}, $R\cong T_1\times \cdots \times T_n$ is a direct product of F-finite rings such that $Z_{T_i}=\Spec(T_i)$. For each $1\leq i\leq n$ let $E_i=T_i[x_1,\cdots ,x_{\gamma(R)-\gamma(T_i)}]$. Observe that $T_i\rightarrow E_i$ is a faithfully flat map of F-finite rings such that $Z_{E_i}=\Spec(E_i)$, with regular fibers, and $\gamma(E_i)=\gamma(R)$. Let $T=E_1\times \cdots \times E_n$ and $R\rightarrow T$ be the natural map. It is easily verified that $Z_T=\Spec(T)$.
\end{proof}

We will use Hochster's and Huneke's gamma constructions to prove Theorem~\ref{Global F-signature non-F-finite} below. We briefly recall some basic properties of gamma constructions, all of which can be found in \cite[Section 6]{HHTAMS}. Suppose that $R$ is essentially of finite type over a complete local ring $(A,\m, k)$. Let $\Lambda$ be a $p$-base for $k$. For each cofinite subset of $\Gamma\subseteq \Lambda$, there is an associated F-finite ring $R^{\Gamma}$ and faithfully flat purely inseparable ring homomorphism $R\rightarrow R^{\Gamma}$. It follows that $\Spec(R^{\Gamma})\rightarrow \Spec(R)$ is a homeomorphism with inverse map $P\mapsto P_\Gamma=\sqrt{PR^{\Gamma}}.$

For every given $P \in \Spec(T)$ there exists a cofinite subset $\Gamma_0 \subseteq \Lambda$ such that $PR^{\Gamma} = P_{\Gamma}$ for all cofinite subsets $\Gamma \subseteq \Gamma_0$. Therefore, for every given $P$ and cofinite $\Gamma_1 \subseteq \Lambda$, there exists a cofinite $\Gamma_2 \subseteq \Gamma_1$ such that $PR^{\Gamma} = P_{\Gamma}$ for all cofinite subsets $\Gamma \subseteq \Gamma_2$.

Suppose that $R$ is essentially of finite type over a complete local ring $(A,\m,k)$. Let $\Lambda$ be a $p$-base for $k$ and let $\Gamma\subseteq \Lambda$ be a cofinite subset. Then for each $P\in\Spec(R)$ we have flat map of local rings $R_P\rightarrow (R^\Gamma)_{P_\Gamma} =: R_{P_\Gamma}^{\Gamma}$. Then $\s(M_P)\geq \s(M \otimes_R R_{P_\Gamma}^{\Gamma})$, with equality if $PR_{P_\Gamma}^{\Gamma}$ is prime, see Theorem~\ref{F-signature comparison}. We remark that it is not necessarily the case that there exists $\Gamma\subseteq \Lambda$ cofinite such that $PR_{\Gamma}$ is prime for every $P\in\Spec(R)$. Hence one cannot necessarily expect to find $\Gamma\subseteq \Lambda$ such that $\s(M_P)=\s(M\otimes_R R^{\Gamma}_{P_\Gamma})$ for all $P\in\Spec(R)$. However, we show in Theorem~\ref{Global F-signature non-F-finite} below that one can find $\Gamma\subseteq \Lambda$ such that $\s(M_P)$ and $\s(M \otimes_S R^{\Gamma}_{P_{\Gamma}})$ are arbitrarily close for all $P\in\Spec{R}$.

\begin{Remark}\label{Assume s>0 remark} Let $R$ be either F-finite or essentially of finite type over an excellent local ring and let $M$ be a finitely generated $R$-module. Assume that $R\rightarrow T$ is faithfully flat and $T$ is F-finite. If $\ts(M)=0$ then it easily follows by Theorem~\ref{F-signature comparison} that $\ts(M\otimes_R T)=0$ and therefore $\s(M\otimes_R T)=0$. If $\ts(M)>0$ then $\ts(R)>0$ and $R$ is strongly F-regular by Theorem~\ref{s>0 local}. In particular, $R$ is locally equidimensional and if $R$ is F-finite, the functions $\s_e:\Spec(R)\rightarrow \R$ sending $P\mapsto \s_e(M_P)$ are lower semi-continuous by \cite[Corollary~2.5 and Remark~5.5]{EnescuYao}. If $R$ is essentially of finite type over an excellent local ring $(A,\m,k)$, then $R\rightarrow \hat{A}\otimes_A R$ is faithfully flat with regular fibers, \cite[Section~33, Lemma~4]{Matsumura1980}. It follows by Theorem~\ref{F-signature comparison} that $\ts(R)=\ts(\hat{A}\otimes_A R)$. In particular, $\hat{A}\otimes_A R$ remains strongly F-regular and therefore locally equidimensional. Hence $\s_e:\Spec(R)\rightarrow \R $ sending $P\mapsto \s_e(M_P)$ is lower semi-continuous by \cite[Theorem~5.1 and Remark~5.5]{EnescuYao}. It follows that if $R$ is strongly F-regular, then for each $e\in\N$ there exists $Q_e\in\Spec(R)$ such that $\s_e(M_{Q_e})=\min\{\s_e(M_P)\mid P\in\Spec(R)\}$.
\end{Remark}

\begin{Theorem}\label{Global F-signature non-F-finite} Let $R$ be either F-finite or essentially of finite type over an excellent local ring and $M$ a finitely generated $R$-module. If $R$ is not strongly F-regular then $\ts(M)=\ts(M\otimes_R T)=\s(M\otimes_R T)=0$ for every faithfully flat F-finite extension $R\rightarrow T$. If $R$ is strongly F-regular then the following limits exist:
\begin{enumerate}
\item $\displaystyle \lim_{e\rightarrow \infty}s_e(M_{Q_e})$, where $Q_e\in \Spec(R)$ is chosen such that $$\s_e(M_{Q_e})=\min\{\s_e(M_P)\mid P\in\Spec(R)\},$$
\item $\displaystyle \lim_{e\rightarrow \infty}\s(M_{Q_e})$, where $Q_e\in \Spec(R)$ is chosen such that $$\s_e(M_{Q_e})=\min\{\s_e(M_P)\mid P\in\Spec(R)\},$$
\end{enumerate}
and they agree with the local-minimal F-signature $\ts(M)$. 
Moreover, 
\begin{align*}
\ts(M)&=\sup\{\ts(M \otimes_R T)\mid R\rightarrow T\mbox{  is faithfully flat and $T$ is F-finite} \}\\
&=\sup\{\s(M \otimes_R T)\mid R\rightarrow T\mbox{  is faithfully flat and $T$ is F-finite} \}.
\end{align*}
Under the assumption that $R$ is F-finite,
\begin{align*}
\ts(M)&=\max\{\ts(M \otimes_R T)\mid R\rightarrow T\mbox{  is faithfully flat and $T$ is F-finite} \}\\
&=\max\{\s(M \otimes_R T)\mid R\rightarrow T\mbox{  is faithfully flat and $T$ is F-finite} \}.
\end{align*}
\end{Theorem}

\begin{proof} By Remark~\ref{Assume s>0 remark} we may assume that $R$ is strongly F-regular. For each $e\in \N$ let $\s_e:\Spec(R)\rightarrow \R$ be the function sending $P\mapsto \s_e(M_P)$ and let $\s:\Spec(R)\rightarrow \R$ be the function mapping $P\mapsto \s(M_P)$. The functions $\s_e$ converge uniformly to $\s$ by Theorem~\ref{Uniform Convergence of a_e} and Remark~\ref{Uniform Convergence Remark}. It follows that the limits in (1) and (2) exist and are equal to $\ts(M)$. See Lemma~\ref{Convergence Lemma} for a similar argument.

Let $R\rightarrow T$ be faithfully flat, with $T$ an F-finite ring. Let $P\in\Spec(R)$ be chosen such that $\ts(M)=\s(M_P)$. As $R\rightarrow T$ is faithfully flat there exists $Q\in T$ such that $Q\cap R=P$. By Theorem~\ref{F-signature comparison} $\s(M_P)\geq \s(M\otimes_R T_Q)$, hence $\ts(M)\geq \ts(M\otimes_R T)$. If $Z_T \neq \Spec(T)$, then $\s(M\otimes_R T) = 0$, see Remark~\ref{Reduce to the Reduced Case}. Else $Z_T=\Spec(T)$ and $\s(M\otimes_R T)=\ts(M\otimes_R T)$ by Theorem~\ref{What is Global F-signature}. This shows
\begin{align*}
\ts(M)&\geq\sup\{\ts(M \otimes_R T)\mid R\rightarrow T\mbox{  is faithfully flat and $T$ is F-finite} \}\\
&\geq\sup\{\s(M \otimes_R T)\mid R\rightarrow T\mbox{  is faithfully flat and $T$ is F-finite} \}.
\end{align*}

Suppose that $R$ is F-finite. We show the existence of a faithfully flat F-finite extension $R\rightarrow T$ such that $\ts(M)=\s(M\otimes_R T)$. Since $R$ is strongly F-regular, we have $R\cong D_1\times \cdots\times  D_n$ is a product of F-finite domains $D_i$  by Theorem~\ref{s>0 local}. By Lemma~\ref{Locally Equidimensional Lemma} there exists a faithfully flat extension $R\rightarrow T$ with regular fibers, $T$ is F-finite, and $Z_T=\Spec(T)$. In particular, $\ts(M)=\ts(M\otimes_R T)$ by Theorem~\ref{F-signature comparison}. As $Z_T=\Spec(T)$ we see that $\s(M\otimes_R T)=\ts(M\otimes_R T)$ by Theorem~\ref{What is Global F-signature}. 

Now suppose that $R$ is essentially of finite type over an excellent local ring $(A,\m,k)$. Let $\epsilon>0$. We are going to show the existence of a faithfully flat extension $R\rightarrow T$ such that $T$ is F-finite and $\s(M\otimes_R T)> \ts(M)-\epsilon$, which will complete the proof of the theorem.  Denote by $\hat{A}$ the completion of $A$ with respect to its maximal ideal. Then $R\rightarrow \hat{A}\otimes_A R$ is faithfully flat with regular fibers, \cite[Section 33, Lemma 4]{Matsumura1980} and, by Theorem~\ref{F-signature comparison}, we have that $\ts(M)=\ts(\hat{A}\otimes_A M)$. Thus we may replace $R$ with $\hat{A}\otimes_A R$ and assume that $R$ is essentially of finite type over a complete local ring. 
 
 Abusing notation, we let $(A,\m,k)$ be a complete local ring which $R$ is essentially of finite type over. Without loss of generality, assume that $\epsilon<\ts(M)$.  Let $\Lambda$ be a $p$-base for a coefficient field $k\subseteq A$. For each cofinite subset $\Gamma\subseteq \Lambda$ let $$U_\Gamma=\{P\in \Spec(R)\mid \s(M\otimes _RR^{\Gamma}_{P_{\Gamma}})>\ts(M)-\epsilon\}.$$ For each $\Gamma$, the induced map of spectra $\Spec(R^\Gamma)\rightarrow \Spec(R)$ is a homeomorphism, hence by Theorem~\ref{Lower semi-continuity of the F-signature} the sets $U_\Gamma$ are open. Moreover, if $\Gamma'\subseteq \Gamma$, then Theorem~\ref{F-signature comparison} shows that $U_{\Gamma'}\supseteq U_{\Gamma}$. As $\Spec(R)$ is Noetherian, there exists some cofinite subset $\Gamma\subseteq \Lambda$ such that $U_{\Gamma}$ is maximal. We claim that $U_{\Gamma}=\Spec(R)$. Else, there exists $P\in\Spec(R)-U_{\Gamma}$. There exists some cofinite subset $\Gamma'\subseteq \Gamma$ such that $PR^{\Gamma'}=P_{\Gamma'}$, i.e., $PR^{\Gamma'}$ is prime. In which case, $R_P\rightarrow R^{\Gamma'}_{P_{\Gamma'}}$ is a faithfully flat local homomorphism whose closed fiber is a field. By Theorem~\ref{F-signature comparison}, $\s(M_P)=\s(M\otimes_R R^{\Gamma'}_{P_{\Gamma'}})$. Therefore $P\in U_{\Gamma'}$, and then $P \in U_{\Gamma}$ by maximality. This contradicts the choice of $P$.  Thus we have $\s(M\otimes_R R^\Gamma_{P_{\Gamma}})> \ts(M)-\epsilon>0$ for all $P_{\Gamma}\in\Spec(R^\Gamma)$, which implies $\ts(M\otimes_R R^\Gamma)>\ts(M)-\epsilon$. In particular, $R^\Gamma$ is strongly F-regular and is a direct product of F-finite domains. By Lemma~\ref{Locally Equidimensional Lemma} there exists faithfully flat F-finite extension $R^\Gamma\rightarrow T$ with regular fibers and such that $Z_T=\Spec(T)$. Hence $\ts(M\otimes_R R^\Gamma)=\ts(M\otimes_R T)$ by Theorem~\ref{F-signature comparison} and $\ts(M\otimes_R T)=\s(M\otimes_R T)$ by Theorem~\ref{What is Global F-signature}. Therefore $\s(M\otimes_R T)>\ts(M)-\epsilon$, which completes the proof.
 \end{proof}
 
\subsection{Global Hilbert-Kunz multiplicity} We now discuss the behavior of global Hilbert-Kunz multiplicity under faithfully flat extensions. Recall that if $R$ is F-finite and $M$ a finitely generated $R$-module then $\e(R)=\max\{\e(R_P)\mid P\in Z_R\}$ by Theorem~\ref{What is HK}.

\begin{Theorem}\label{HK Flat Extension Global} Let $R\rightarrow T$ be a faithfully flat extension of F-finite rings and let $M$ be a finitely generated $R$-module. If each $P\in Z_R$ is a contraction of a prime $Q\in Z_T$, then $\e(M)\leq \e(M\otimes_R T)$. In particular, if $R$ and $T$ are domains, or more generally if $R$ and $T$ are such that $Z_R=\Spec(R)$ and $Z_T=\Spec(T)$, then $\e(M)\leq \e(M\otimes_R T)$ with equality if the closed fibers of $R\rightarrow T$ are regular.
\end{Theorem}

\begin{proof} By Theorem~\ref{What is HK}, $\e(M)=\max\{\e(M_P)\mid P\in Z_R\}$ and $\e(M\otimes_R T)=\max\{\e(M\otimes_R T_Q)\mid Q\in Z_T\}$. Let $P \in Z_R$ be such that $\e(M) = \e(M_P)$. By assumption, there exists $Q \in Z_T$ such that $Q \cap R = P$. By Theorem~\ref{HK Inequality} we obtain that $\e(M) = \e(M_P) \leq \e(M \otimes_R T_Q) \leq \e(M \otimes_R T)$. Now suppose $Z_R=\Spec(R)$, $Z_T=\Spec(T)$, and the closed fibers of $R \to T$ are regular. Then there exists $\n \in \Max(T)$ such that $\e(M \otimes_R T) = \e(M \otimes_R T_\n)$.
Let $\m$ be the contraction of $\n$ in $R$, then $R_\m\rightarrow T_\n$ is flat with regular fiber. By Theorem~\ref{HK Inequality}, $\e(M_\m)=\e(M \otimes_R T_\n)=\e(M \otimes_R T)$. The theorem follows since $\e(M)\geq \e(M_\m)$. 
\end{proof}

\begin{Example}\label{global HK can decrease after fflat extension} For an arbitrary faithfully flat extension $R\rightarrow T$ of F-finite rings, it need not be the case that $\e(R)\leq \e(T)$. Suppose that $R$ is an F-finite domain such that $\e(R)>1$. Let $S=K[x]$ where $K$ is the fraction field of $R$. Take $T$ to be the direct product of rings $R\times S$. Then the natural map $R\rightarrow T$ is faithfully flat and $\e(T)=1<\e(R)$.
\end{Example}

\begin{Example}\label{Faithfully flat example revisted} If $R\rightarrow T$ is a faithfully flat map of F-finite domains, then it need not be the case that $\mu(F^e_*R)/p^{e\gamma(R)}\leq \mu(F^e_*T)/p^{e\gamma(T)}$, even though the inequality holds after taking limits. One should compare this to the local situation in Theorem~\ref{HK Flat Extension}. In fact, the same example used in Example~\ref{Faithfully flat example} demonstrates such phenomena. Suppose that $R$ is a Dedekind domain affine over an algebraically closed field $k$ of characteristic $p$. Then $F^e_*R$ is projective of rank $p^e$. Hence by Theorem~\ref{Forster-Swan}, $\mu(F_*R)$ is either $p^e$ or $p^{e}+1$. The case that $\mu(F_*R)=p^e$ corresponds to the case that $F^e_*R$ is free and $\mu(F^e_*R)=p^{e}+1$ corresponds to the case that $F^e_*R$ is not free. Suppose that $R$ is as in Example~\ref{Flat but not free}, that is $F^e_*R$ is not free. Consider the faithfully flat extension $R\rightarrow R[t]\rightarrow T=R[t]_W$ where $W$ is the multiplicative set $R[t]-\cup_{\m\in\Max(R)}\m R[t]$. Then $T$ is a Dedekind domain and $F^e_*T$ is a projective $T$-module of rank $p^{2e}$. By Theorem~\ref{Forster-Swan}, $\mu(F^e_*T)$ is either $p^{2e}$ or $p^{2e}+1$. But $\mu(F^e_*R)/p^{e\gamma(R)}=\frac{p^e+1}{p^e}>\frac{p^{2e}+1}{p^{2e}}\geq \mu(F^e_*T)/p^{e\gamma(T)}$.
\end{Example} 

Suppose that $R$ is either F-finite or essentially of finite type over an excellent local ring, and $M$ is a finitely generated $R$-module. We defined $\ts(M)$ and showed in Theorem~\ref{Global F-signature non-F-finite} that if $R\rightarrow T$ is faithfully flat and $T$ is F-finite, then $\ts(M)\geq \ts(M\otimes_R T)\geq \s(M\otimes_R T)$. Moreover, for $\epsilon >0$, there exists $R\rightarrow T$ faithfully flat and F-finite such that $\ts(M)<\s(M\otimes_R T)+\epsilon$. We now develop an analogous theory for Hilbert-Kunz multiplicity.

Define the \emph{local-maximal Hilbert-Kunz multiplicity of $M$} to be \[\te(M)=\sup\{\e(M_P)\mid P\in\Spec(R)\}.\] As the Hilbert-Kunz multiplicity function is not upper semi-continuous without the locally equidimensional hypothesis, there may not be a prime $P\in\Spec(R)$ such that $\te(M)=\e(M_P)$. Suppose that $R\rightarrow T$ is faithfully flat and $T$ is F-finite. It easily follows by Theorem~\ref{HK Inequality} that $\te(M)\leq \te(M\otimes_R T)$. However, it may be the case that $\te(M\otimes_R T)>\e(M\otimes_R T)$ or it may be the case that there is faithfully flat $T\rightarrow T'$ such that $T'$ is F-finite and $\e(M\otimes_R T)>\e(M\otimes_R T')$, see Example~\ref{global HK can decrease after fflat extension}. Nevertheless, we can still develop an analogue of Theorem~\ref{Global F-signature non-F-finite} for Hilbert-Kunz multiplicity, but under appropriate hypotheses.

\begin{Theorem}\label{Global Hilbert-Kunz non-F-finite} Let $R$ be a locally equidimensional ring which is either F-finite or essentially of finite type over an excellent local ring $(A,\m,k)$. Let $M$ be a finitely generated $R$-module.  Then the following limits exist:
\begin{enumerate}
\item $\displaystyle \lim_{e\rightarrow \infty}\lambda(M_{Q_e}/Q_e^{[p^e]}M_{Q_e})/p^{e\height(Q_e)}$, where $Q_e\in \Spec(R)$ is chosen such that $$\lambda(M_{Q_e}/Q_e^{[p^e]}M_{Q_e})/p^{e\height(Q_e)}=\max\{\lambda(M_{P}/P^{[p^e]}M_{P})/p^{e\height(P)}\mid P\in\Spec(R)\},$$
\item $\displaystyle \lim_{e\rightarrow \infty}\e(M_{Q_e})$, where $Q_e\in \Spec(R)$ is chosen such that $$\lambda(M_{Q_e}/Q_e^{[p^e]}M_{Q_e})/p^{e\height(Q_e)}=\max\{\lambda(M_{P}/P^{[p^e]}M_{P})/p^{e\height(P)}\mid P\in\Spec(R)\}.$$
\end{enumerate}
All of the above limits agree, with common value being the local-maximal Hilbert-Kunz multiplicity $\te(M)$. Under the assumption that $R$ is F-finite,
\begin{align*}
\te(M)&=\min\{\te(M \otimes_R T)\mid R\rightarrow T\mbox{  is faithfully flat and $T$ is F-finite} \}\\
&=\min\{\e(M \otimes_R T)\mid R\rightarrow T\mbox{  is faithfully flat, $T$ is F-finite, and }Z_T=\Spec(T) \}.
\end{align*}
In the case that $R$ is essentially of finite type over an excellent local ring $(A,\m,k)$ such that $\hat{A}\otimes_A R$ is locally equidimensional,
\begin{align*}
\te(M)&=\inf\{\te(M \otimes_R T)\mid R\rightarrow T\mbox{  is faithfully flat and $T$ is F-finite} \}\\
&=\inf\{\e(M \otimes_R T)\mid R\rightarrow T\mbox{  is faithfully flat, $T$ is F-finite, and }Z_T=\Spec(T) \}.
\end{align*}
\end{Theorem}

\begin{proof}Similar to the proof of Theorem~\ref{Global F-signature non-F-finite}, the existence of the limits in (1) and (2) and their convergence to $\te(M)$ is a statement about the uniform limit of semi-continuous functions defined on a quasi-compact topological space. See Theorem~\ref{Uniform Convergence of HK} and Theorem~\ref{Upper Semi-Continuity of HK} for the necessary details.

Suppose that $R\rightarrow T$ is faithfully flat and $T$ is F-finite. It easily follows by Theorem~\ref{HK Inequality} that $\te(M)\leq \te(M\otimes_R T)$. Moreover, if $Z_T=\Spec(T)$ then $\te(M\otimes_R T)=\e(M\otimes_R T)$ by Theorem~\ref{What is HK}. Therefore
\begin{align*}
\te(M)&\leq\inf\{\te(M \otimes_R T)\mid R\rightarrow T\mbox{  is faithfully flat and $T$ is F-finite} \}\\
&\leq\inf \{\e(M \otimes_R T)\mid R\rightarrow T\mbox{  is faithfully flat, $T$ is F-finite, and }Z_T=\Spec(T) \}.
\end{align*}

Suppose that $R$ is F-finite, we show the existence of a faithfully flat extension $R\rightarrow T$ such that $\te(M)=\e(M\otimes_R T)$. We are assuming that $R$ is locally equidimensional. Let $T$ be as in Lemma~\ref{Locally Equidimensional Lemma}, that is $R\rightarrow T$ is faithfully flat, with regular fibers, $T$ is F-finite, and $Z_T=\Spec(T)$. By Theorem~\ref{HK Inequality} and Theorem~\ref{What is HK}, $\te(M)=\te(M\otimes_R T)=\e(M\otimes_R T)$.

Now suppose that $R$ is essentially of finite type over an excellent local ring $(A,\m,k)$ and $\hat{A}\otimes_A R$ is locally equidimensional. Let $\epsilon>0$. We are going to show the existence of a faithfully flat extension $R\rightarrow T$ such that $T$ is F-finite, $Z_T=\Spec(T)$, and such that $\te(M)+\epsilon>\te(M\otimes_R T)$. The extension $R\rightarrow \hat{A}\otimes_A R$ is faithfully flat with regular fibers, \cite[Section 33, Lemma 4]{Matsumura1980}. By Theorem~\ref{HK Inequality}, $\te(M)=\te(\hat{A}\otimes_A M)$, hence we may replace $R$ by $\hat{A}\otimes_A R$ and $M$ by $\hat{A}\otimes_A M$ and assume $R$ is equidimensional and essentially of finite type over a complete local ring $(A,\m,k)$.

 Let $\Lambda$ be a $p$-base for $k$. For each cofinite subset $\Gamma\subseteq \Lambda$, $\Spec(R)$ is homeomorphic to $\Spec(R^\Gamma)$, and therefore the F-finite ring $R^\Gamma$ is locally equidimensional. For each cofinite subset $\Gamma\subseteq \Lambda$, let $$U_\Gamma=\{P\in \Spec(R)\mid \e(M\otimes_R R^{\Gamma}_{P_\Gamma})<\te(M)+\epsilon\}.$$ Then $U_\Gamma$ is open by Theorem~\ref{Upper Semi-Continuity of HK}. Moreover, if $\Gamma'\subseteq \Gamma$ then $U_{\Gamma}'\supseteq U_{\Gamma}$ by Theorem~\ref{HK Inequality}. By Noetherian property, there exists a cofinite set $\Gamma\subseteq \Lambda$ such that $U_\Gamma$ is maximal. If there exists $P\in\Spec(R)-U_\Gamma$, then choose $\Gamma'\subseteq \Gamma$ such that $PR^{\Gamma'}=P_{\Gamma'}$. Theorem~\ref{HK Inequality} implies $P\in U_{\Gamma'}$ contradicting maximality of $U_\Gamma$. It readily follows that $\te(M)+\epsilon> \te(M\otimes_R R^\Gamma)$. As $R^\Gamma$ is locally equidimensional, there exists faithfully flat extension $R^\Gamma\rightarrow T$ with regular fibers such that $T$ is F-finite and $Z_T=\Spec(T)$, by Lemma~\ref{Locally Equidimensional Lemma}. We have $\te(M\otimes_R R^\Gamma)=\te(M\otimes_R T)$ by Theorem~\ref{HK Inequality} and $\te(M\otimes_R T)=\e(M\otimes_R T)$ by Theorem~\ref{What is HK}. Therefore $\te(M)+\epsilon> \e(M\otimes_R T)$, which completes the proof of the theorem. \end{proof} 

\section*{Acknowledgments}

The authors of this paper are grateful for numerous fruitful conversations we have had with several mathematicians concerning this project. In particular, we would like to thank Ian Aberbach, Craig Huneke, Luis N\'{u}\~{n}ez-Betancourt, Karl Schwede, and Kevin Tucker.

\bibliographystyle{alpha}
\bibliography{References}

\end{document}